\documentclass[preprint,11pt]{elsarticle}
\usepackage{graphicx} 

\usepackage[left=3cm,top=3cm,bottom=3cm,right=3cm]{geometry}

\usepackage[english]{babel}      
\usepackage[utf8]{inputenc}      
\usepackage{microtype}           
\usepackage[T1]{fontenc}         
\usepackage{lmodern}             
\usepackage{textcomp}            
\usepackage{verbatim}            
\usepackage{placeins}            
\usepackage{color}               
\usepackage{graphicx}            
\usepackage{nicefrac}            
\usepackage[hidelinks]{hyperref} 
\usepackage{todonotes}           
\usepackage{overpic}
\usepackage{wrapfig}             
\usepackage{subfigure}           
\usepackage{ragged2e}            
\usepackage{float}               
\usepackage{latexsym,exscale,stmaryrd,amssymb,amsmath}
\usepackage{mathtools}
\usepackage{bm}
\usepackage{dsfont}

\usepackage{amsthm}

\usepackage[nameinlink]{cleveref}

\usepackage{thmtools}
\usepackage{cleveref}
\usepackage{booktabs}

\declaretheoremstyle[
spaceabove=6pt, spacebelow=6pt,
headfont=\normalfont\bfseries,
notefont=\mdseries, notebraces={(}{)},
bodyfont=\normalfont,
postheadspace=\newline,
qed={}
]{mystyle}

\newtheoremstyle{break} 
  {\topsep}{\topsep}%
  {\itshape}{}%
  {\bfseries}{}%
  {\newline}{}

\theoremstyle{break}
\newtheorem{thm}{Theorem}[section]

\theoremstyle{break} 

\newtheorem*{defn*}{Definition} 

\theoremstyle{remark}
\newtheorem{rem}[thm]{Remark}

\declaretheorem[name=Theorem,refname={Theorem,Theorems},Refname={Theorem,Theorems},style=mystyle,numberlike=thm]{xthm}

\usepackage{siunitx}
\DeclareSIUnit{\skalenteil}{Skt}

\usepackage{mathrsfs} 

\usepackage{enumerate}

\usepackage{stackrel}

\renewcommand{\div}{\operatorname{div}}
\usepackage{tikz}
\usetikzlibrary{cd}

\newtheorem{theorem}{Theorem}[section]

\theoremstyle{remark}
\newtheorem{remark}[theorem]{Remark}
\theoremstyle{definition}

\newtheoremstyle{algobreak}
  {\topsep}{\topsep}%
  {\normalfont}{}
  {\bfseries}{}
  {\newline}{}

\theoremstyle{algobreak}
\newtheorem{algo}[thm]{Algorithm} \DeclareMathOperator{\rot}{\mathbf{rot}}
\DeclareMathOperator{\grad}{\mathbf{grad}}
\renewcommand{\div}{\operatorname{div}}

\newcommand{\BDM}{\mathds{BDM}}
\newcommand{\JBDM}{\mathds{J}^k_{\mathds{BDM}}}
\newcommand{\HBDM}{\mathds{H}^k_{\mathds{BDM}}}
\newcommand{\FacetSp}{\bm{\Lambda}_h^k}
\newcommand{\SFacetSp}{\Lambda_h^k}

\newcommand{\R}{\mathbb{R}}

\newcommand{\E}{\mathcal{E}}

\newcommand{\Th}{\mathcal{T}_h}

\newcommand{\Eh}{\mathcal{E}_h}

\renewcommand{\div}{\operatorname{div}}
\newcommand{\addiv}{\mathbf{div}^*}
\newcommand{\curl}{\operatorname{curl}}

\newcommand{\scp}[2]{\left( #1,#2\right)}
\newcommand{\norm}[1]{\left\|#1\right\|}

\newcommand{\nv}{\bm{\nu}}
\newcommand{\tv}{\bm{\tau}}
\newcommand{\TM}{\mathbf{T}M}
\newcommand{\tend}{T}

\newcommand{\dirder}[1]{\partial_{#1} }

\newcommand{\Hn}{\mathbf{H}_N(M)}

\newcommand{\jump}[1]{[\![ #1 ]\!]}
\newcommand{\avg}[1]{\{\!\!\{#1\}\!\!\}}

\renewcommand{\u}{\bm{u}}
\renewcommand{\v}{\bm{v}}
\newcommand{\w}{\bm{w}}

\newcommand{\urot}{\bm{u}_{\scalebox{0.6}{$\rot$}}}
\newcommand{\uharm}{\bm{u}_{\scalebox{0.6}{$\mathds{H}$}}}

\DeclareMathOperator{\ran}{ran}

\newcommand{\dofs}{\texttt{dofs}}

\newcommand{\osum}{\oplus_{\scalebox{0.6}{$L^2$}}}

\begin{document}

\begin{frontmatter}

\title{Releasing the pressure: High-order surface flow discretizations via discrete Helmholtz--Hodge decompositions}

\date{\today}

\author[1]{Tim Brüers}
\ead{t.brueers@math.uni-goettingen.de}
\author[1]{Christoph Lehrenfeld}
\ead{lehrenfeld@math.uni-goettingen.de}
\author[1]{Tim van Beeck}
\ead{t.beeck@math.uni-goettingen.de}
\author[1]{Max Wardetzky}
\ead{wardetzky@math.uni-goettingen.de}

\affiliation[1]{organization={Institute for Numerical and Applied Mathematics, University of G\"ottingen},
            addressline={Lotzestr. 16-18},
            city={G\"ottingen},
            postcode={37083},
            country={Germany}}

\begin{abstract}
We present a discrete Helmholtz--Hodge decomposition for $H(\div)$-conforming Brezzi--Douglas--Marini (BDM) finite elements on triangulated surfaces of arbitrary topology.
The divergence-free BDM subspace is split $L^2$-orthogonally into rotated gradients of a continuous streamfunction space and a finite-dimensional space of discrete harmonic fields whose dimension equals the first Betti number of the surface.
Consequently, any incompressible flow discretized on
this subspace can be reformulated with a scalar streamfunction and finitely many harmonic
coefficients as the only unknowns.
This eliminates the pressure and the saddle-point structure while ensuring exact tangentiality, pointwise divergence-freeness, and pressure-robustness.
We present a randomized algorithm for constructing the harmonic basis and discuss implementation aspects including hybridization, efficient treatment of the harmonic unknowns, and pressure reconstruction.
Numerical experiments for unsteady surface Navier--Stokes equations on a trefoil knot and a multiply-connected sculpture surface demonstrate the method and illustrate the physical role of the harmonic velocity component.
\end{abstract}


\end{frontmatter}

\section{Introduction}\label{sec:intro}
Incompressible flows on curved surfaces arise in thin-film dynamics, in the modeling of biomembranes, and in geophysical applications. By \emph{surfaces} we here refer to orientable, compact, two-dimensional Riemannian manifolds, possibly with boundary. Incompressible flows on such surfaces are characterized by two essential geometric constraints: the velocity must be tangential to the surface and divergence-free.
Numerically, both of these conditions can be satisfied using $\mathbf{H}(\div)$-conforming Brezzi--Douglas--Marini ($\BDM$) finite elements, which yield pointwise divergence-free velocities,
pressure-robustness~\cite{linke2014role,LehrenfeldSchoeberl2016}, and energy
stability~\cite{LedererLehrenfeldSchoeberl2020}. 
Nonetheless, the resulting formulation still requires solving a velocity-pressure saddle-point system.

The key observation of this work is that such a saddle-point approach can be avoided when using a discrete analogue of the smooth $L^{2}$‑orthogonal Helmholtz--Hodge
decomposition of divergence-free fields into a streamfunction rotation and a harmonic
field~\cite{StreamVortForm}. 
Indeed, this decomposition has a direct discrete counterpart for the space $\BDM^k$, the $\BDM$ space of degree $k$. 
On a surface $M$ with first Betti number $b_{1}(M)$, the discrete divergence-free subspace $\JBDM \subset \BDM^k$ splits
\begin{align}\label{eq:intro:HHD}
    \JBDM = \rot(\mathbb{S}^{k+1}_0)\;\osum\;\HBDM,
\end{align}
where $\mathbb{S}^{k+1}_0$ is the continuous Lagrange finite element space of degree $k+1$, referred to as the streamfunction space, and $\HBDM$ is a $b_{1}(M)$-dimensional space of discrete harmonic fields.

This decomposition accounts for the topology of the underlying surface: harmonic fields capture possible non-contractible circulations that cannot be represented by any streamfunction. Indeed, while for simply connected surfaces the velocity field perfectly aligns with stream function contours, 
%
the harmonic component carries dynamically relevant information for surfaces of higher genus, such as the net flow through non-contractible loops --- see \Cref{fig:streamlinesnotrots}.

\begin{figure}[!htbp]
    \centering
    \vspace*{-0.05cm}
    \includegraphics[width=0.8\textwidth]{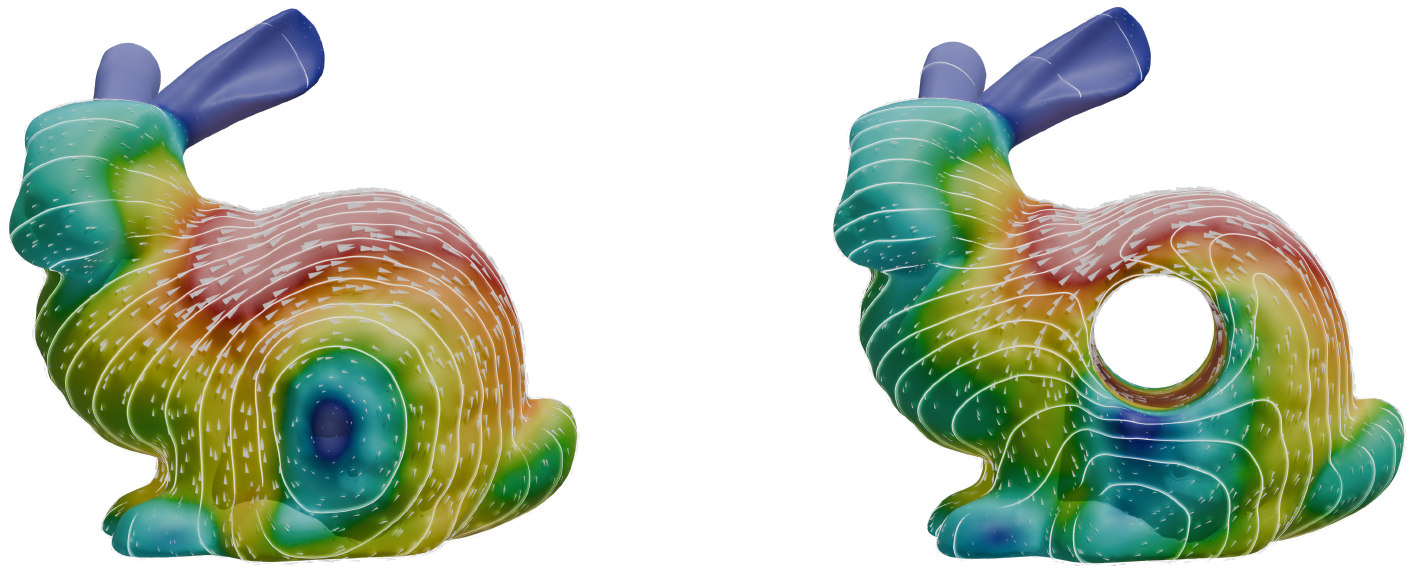}
    \vspace*{-0.45cm}
    \caption{Comparison of the standard and ``hol(e)y'' Stanford Bunny. Both surfaces depict a divergence-free vector field alongside contours of the corresponding scalar streamfunction. 
    The left bunny is simply connected, and thus no harmonic fields are present. Consequently, the vector field aligns perfectly with the contours of the streamfunction. In contrast, the right bunny has the topology of a torus, and thus the harmonic space has dimension $b_1(M)=2$. Here, the vector field contains a non-trivial harmonic component and therefore clearly deviates from the contours of the streamfunction.}
    \label{fig:streamlinesnotrots}
\end{figure}

Using the decomposition \eqref{eq:intro:HHD}, we can reformulate any incompressible flow problem on $\JBDM$ with the streamfunction $\psi_h \in \mathbb{S}^{k+1}_0$ and $b_1(M)$ harmonic coefficients (that determine a unique harmonic component $\mathbf{h}_h \in \HBDM$) as the only unknowns.
Using this approach, velocity is recovered exactly via $\u_h = \rot(\psi_h) + \bm h_h$ -- and inherits all structural properties of the $\mathbf{H}(\div)$-conforming framework, namely, exact tangentiality, pointwise divergence-freeness, and pressure-robustness (up to geometry and quadrature errors, cf. \cite[Sec. 4.2]{LedererLehrenfeldSchoeberl2020}), without solving a saddle-point problem.

In the special case of a simply connected flat domain, the harmonic space vanishes and the approach reduces to a discrete formulation of the biharmonic equation; see, e.g.~\cite{KanschatSharma2014}.
The general surface case treated here requires additional ingredients: a construction of the harmonic basis within $\JBDM$, and a practical realization of the splitting \eqref{eq:intro:HHD} within existing finite element frameworks. We address both aspects in detail. 

Our main contribution is to show that the discrete velocity space $\BDM$ itself inherits a discrete Helmholtz--Hodge structure that precisely mirrors the continuous one.
This allows us to obtain a structure-preserving method that retains all geometric exactness properties of the high-order $\BDM$ approach, i.e., exact tangentiality, pointwise divergence-freeness, and pressure-robustness, while simultaneously avoiding the velocity-pressure saddle-point system entirely.

\paragraph{Related work}
Finite element methods for incompressible surface flows have received increasing attention in recent years and have been developed along several lines. 
Constructing $\mathbf{H}^1$-conforming tangential vector fields on surfaces is not straightforward. Therefore, approaches based on $\mathbf{H}^1$-conforming elements are  mostly based on $[H^1]^3$-conforming elements and enforce tangentiality weakly. We refer to \cite{Reuther2018,Fries2018,Hansbo2020} for surface finite element methods or \cite{Gross2018,brandner2022finite} for similar approaches for TraceFEM. 

In contrast, $\mathbf{H}(\div)$-conforming approaches on surfaces achieve exact tangentiality via the Piola transformation and were developed for example in \cite{bonito2020divergence,LedererLehrenfeldSchoeberl2020}.  
%
The recent works by Demlow and Neilan~\cite{demlowneilan2024,demlowneilan2025} and Kone, Neilan, and Polling \cite{KNP26} also use Piola maps and result 
in $\mathbf{H}^1$-non-conforming but penalty-free methods with tangential velocities.
In \cite{john2025divergencefreedecoupledfiniteelement}, John, Li and Merdon recently presented an approach decoupling velocity from pressure for a Raviart--Thomas-stabilized Scott--Vogelius formulation on flat domains, while Nochetto and Shakipov~\cite{nochetto2025surfacestokesinfsupcondition} reformulated the Surface Stokes problem into a saddle-point-free coupled problem.

Streamfunction (and vorticity) formulations provide a natural alternative that avoids the pressure variable.
In the flat case, divergence-free $\BDM$ subspaces have been constructed via streamfunctions in \cite{WWY09,mu2018discrete}. 
Furthermore, the equivalence of $\BDM$ interior penalty (IP) discontinuous Galerkin (DG) methods with $C^0$-IP for biharmonic discretizations was established in \cite{KanschatSharma2014}; see also \cite{BrennerSung2005,BrennerNeilan2011,EGHLM02} for $C^0$-IP methods and \cite{DE24C0HHO} for a hybridized variant.
On surfaces, streamfunction formulations for the Stokes problem have been studied in (among others) in~\cite{reusken2020stream,brandner2022finite,NW25} but only for simply connected surfaces, where the absence of harmonic fields simplifies the theory considerably. The comprehensive continuous theory for streamfunction-vorticity formulations on general surfaces, including topologically non-trivial cases, is developed in \cite{StreamVortForm}, proving equivalence with the velocity formulation for both Navier--Stokes and Euler equations.
Discrete exterior calculus approaches to surface Navier--Stokes that handle harmonic fields appear in~\cite{Nitschke2017}.
The decomposition \eqref{eq:intro:HHD} fits naturally into the framework of finite element exterior calculus \cite{AFW06,AFW10}, and the role of incomplete discrete Helmholtz--Hodge decompositions has been analyzed in \cite{BringmannKettelerSchedensack2023}.
The fluid cohomology framework of Yin et al.~\cite{YNWWC23} and its viscous extension~\cite{ZhuYinChern2025} provide a discrete Helmholtz--Hodge decomposition for surface flows from a geometry processing perspective.

%

\paragraph{Outline}
After introducing notation, function spaces, and model problems in \Cref{sec:prelim}, we develop the discrete Helmholtz--Hodge decomposition for $\BDM$ elements on general surface triangulations in \Cref{sec:harmonic}. As a main result, we establish the splitting \eqref{eq:intro:HHD} with $\dim(\HBDM) = b_1(M)$, showing that the discrete chain complex preserves the cohomology of the continuous one. We also discuss discrepancies between the continuous and discrete formulations.
We then outline implementation details in~\Cref{sec:ref-hdiv-dg}, showing how to compute an orthonormal basis of $\HBDM$ using a randomized algorithm and how to formulate the Stokes and Navier--Stokes equations directly in the divergence-free subspace $\JBDM$.
\Cref{sec:implementation} covers further implementation aspects, including hybridization, a Schur complement approach to treat the harmonic unknowns, and pressure reconstruction via post-processing.
Finally, \Cref{sec:numerics} presents numerical experiments for the unsteady surface Navier--Stokes equations on topologically non-trivial surfaces, illustrating the physical role of the harmonic velocity component. We provide supplementary material in the appendices.
 
\section{Preliminaries and notation}\label{sec:prelim}
This section fixes notation for surfaces, meshes, basic polynomial spaces, and surface differential operators. We also state the considered model problems on surfaces.

\subsection{Geometric setup and function spaces}

Let $M \subset \mathbb{R}^3$ be a smooth, oriented, compact, two-dimensional Riemannian manifold, possibly possessing a $C^{0,1}$ boundary $\Gamma$ (as detailed in \cite[App A]{StreamVortForm}) with tangential bundle $\TM$.
We denote by $\bm{\nu}_\Gamma$ the (in-plane) unique outward-pointing unit normal vector on $\Gamma = \partial M$.

On this surface, we consider the standard surface gradient $\grad$ and divergence $\div$. 
Denoting by $J$ a rotation by $90^\circ$, we define the differential operators
\begin{equation*}
    \rot \coloneqq -J\grad \quad \text{and} \quad \curl \coloneqq -\div J.
\end{equation*}
Further, we set $\underline{\nabla}$ as the covariant derivative and the directional derivative in direction $\v$ of a (tangential) smooth vector field $\u$ as
$\dirder{\v} \u \coloneqq (\underline{\nabla} \u ) \cdot \v$ as well as $\nabla^2 = \underline{\nabla} \grad$.
We use standard notation for scalar and vector-valued \emph{tangential} Sobolev spaces $H^m(M)$ and $\mathbf{H}^m(M)$, respectively. 
Similarly, we denote the corresponding subspaces with vanishing trace on the boundary $\Gamma=\partial M$ by $H^m_0(M)$ and $\mathbf{H}^m_0(M)$ for $m \ge 1$. 
In the case of a closed surface where $\Gamma = \emptyset$, we conventionally define for the space $H^m_0(M)$ as the subspace of functions in $H^m(M)$ with zero integral mean. In both cases, we denote by $L^2_0(M)$ the corresponding subspace of functions in $L^2(M)$ with zero integral mean.
Of particular interest for fluid dynamics is the space $\mathbf{H}(\div, M)$, which consists of all vector fields in $\mathbf{L}^2(M)$ whose weak surface divergence belongs to $L^2(M)$. From this space, we define the following divergence-free subspace:
\begin{align*}
    \mathbf{J}_{L^2}(M)&\coloneqq\{\v\in \mathbf{H}(\div,M)\vert~\div(\v)=0~\mathrm{and}~\v\cdot\bm{\nu}_{\Gamma}=0 \text{ in } H^{-\frac12}(\Gamma)\} \subset \mathbf{H}(\div,M).
\end{align*}
We note that $\mathbf{J}_{L^2}(M)$ and its subspaces play a fundamental role in the variational formulation of fluid problems and are thus collectively referred to as velocity spaces.

\begin{figure}[!htbp]
    \centering
    \includegraphics[width=0.8\textwidth]{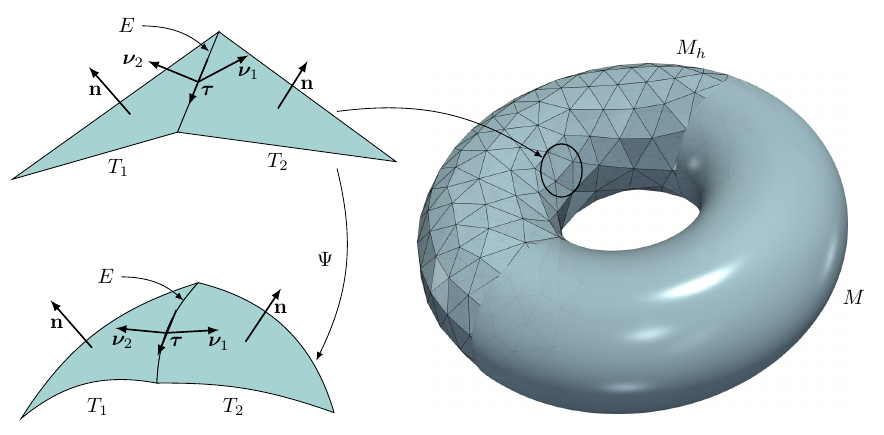}
    \vspace*{-0.225cm}
    \caption{We approximate a surface $M \subset \mathbb{R}^3$ by a discrete surface $M_h$. The approximation of the geometry may be linear (above left) or of higher order (below left) where the latter is typically obtained from the former by a piecewise smooth map $\Psi$.}
    \label{fig:sketchofedgepatch}
\end{figure}

\subsection{Discrete surfaces}\label{subsec:triangulation}
Let $\Th$ be an admissible shape-regular triangulation of the surface $M$ and $\E_h$ the set of all edges. The triangles may be curved with geometric approximation order typically matching the finite element polynomial order. Given a triangulation of closed triangles $T \in \Th$ , we denote by $M_h \coloneqq \cup_{T \in \Th} T$ 
the associated piecewise smooth discrete surface with boundary $\Gamma_h \coloneqq \partial M_h$, carrying the induced metric of the ambient space $\mathbb{R}^3$.
For each element $T \in \Th$, we denote by $\bm{n}\vert_{T}$ the oriented unit normal field on $M_h$. 
Elementwise, this allows us to rewrite $J$, the rotation by $90^\circ$ on $M_h$, as $J \cdot = \bm{n} \times \cdot$.
For any two neighboring elements $T_1,T_2 \in \Th$ with shared edge $E \coloneqq T_1 \cap T_2$, we consider an oriented unit vector field $\bm{\tau}$ along $E$ and define the in-plane unit normal vector fields 
\begin{align}
    \nv_1 \coloneqq \bm{n} \vert_{T_1} \times \bm{\tau}, \qquad \nv_2 \coloneqq - \bm{n} \vert_{T_2} \times \bm{\tau}.
\end{align}
This geometric setup is visualized in \Cref{fig:sketchofedgepatch}. On the edge $E$, we define for a matrix-valued map $\underline{\bm{\sigma}}$ the average on $E$ and for a vector-valued map $\u$ the jump on $E$ as 
\begin{align}
    \avg{\underline{\bm{\sigma}} \nv} \coloneqq \frac{1}{2} \left( \underline{\bm{\sigma}} \vert_{T_1} \nv_1 - \underline{\bm{\sigma}} \vert_{T_2} \nv_2 \right), \qquad  \jump{\u} \coloneqq \jump{\u}_{\tv} \tv + \jump{\u}_{\nv} \overline{\nv}, \qquad \overline{\nv} \coloneqq \frac{1}{2} \left( \nv_1 + \nv_2 \right), 
\end{align}
respectively, where the normal and the tangential jump are defined as 
\begin{align}
    \jump{\u}_{\nv} \coloneqq \u \vert_{T_1} \cdot \nv_1 + \u \vert_{T_2} \cdot \nv_2, \qquad \jump{\u}_{\tv} \coloneqq \left( \u \vert_{T_1} - \u \vert_{T_2} \right) \cdot \tv.
\end{align}
On the surface boundary $\Gamma_h$, we set $\jump{\u}_{\bm{\nu}} \coloneqq \u \cdot \bm{\nu}_\Gamma$, $\jump{\u}_{\bm{\tau}} \coloneqq \u \cdot \bm{\tau}$ and $\avg{\underline{\bm{\sigma}}\bm{\nu}} \coloneqq \underline{\bm{\sigma}} \bm{\nu}_{\Gamma}$.
Notice that these definitions  solely rely on an element-local point of view, such that the flux averages $\avg{\underline{\bm{\sigma}} \nv}$ and normal jumps $\jump{\u}_{\nv}$ make sense even on edges across which the normal $\bm{n}$ is discontinuous.

To every element $T \in \Th$ we associate 
scalar and vectorial polynomial spaces through mappings $\Phi_T$ to an associated flat triangle $\hat{T} \subset \mathbb{R}^2$ so that $T = \Phi_T(\hat{T})$.
Let $\hat{\mathbb{P}}^k(\hat{T})$ be the space of polynomials up to degree $k$ on $\hat{T}$. We define the mapped scalar polynomial space as
\begin{subequations}\label{eq:polynomialSpaces}
    \begin{align}
    \mathbb{P}^k(T) &\coloneqq \{ \hat{v} \circ \Phi_T^{-1} \mid\, \hat{v} \in \hat{\mathbb{P}}^k(\hat{T}) \}. \label{eq:ScalarpolynomialSpaces}
    \intertext{
        Similarly, we define $\mathbb{P}^k(E)$ as the mapped polynomial space on an edge $E\in \Eh$.
        To define vector polynomials, we apply the Piola transformation to ensure tangentiality:
    }
    \mathds{P}^k(T) &\coloneqq \{ \mathcal{J}_T^{-1} (D \Phi_T \hat{\v}) \circ \Phi_T^{-1} \mid\, \hat{\v} \in [\hat{\mathbb{P}}^k(\hat{T})]^2 \}. 
    \label{eq:piola}
    \intertext{
        Here, $\mathcal{J}_T \coloneqq (\text{det}(D \Phi_T^T D \Phi_T))^{1/2}$ denotes the ratio of (surface) measures 
        between $T$ and $\hat{T}$.
        Correspondingly, we define $\mathbb{P}^k=\mathbb{P}^k(\Th)$ and $\mathds{P}^k=\mathds{P}^k(\Th)$ as the piecewise mapped polynomial spaces of degree $k$, discontinuous across element boundaries.
        We further define
    }
    \mathbb{P}_{\!\mathcal{J}}^k(T) &\coloneqq \{ \mathcal{J}_T^{-1} \hat{v} \circ \Phi_T^{-1} \mid\, \hat{v} \in \hat{\mathbb{P}}^k(\hat{T}) \} ,
    \label{eq:pressurep}
    \end{align}
\end{subequations}
which only differs from $\mathbb{P}^k(T)$ for curved elements. Further, we have $\div \mathds{P}^k(T) \subset \mathbb{P}_{\!\mathcal{J}}^{k-1}(T)$ \cite[Lem.~3.59]{Monk03}.
Finally, by $\Pi_{S}$ we denote the $L^2$-projection into the generic function space $S$. 

\subsection{Model problems on surfaces}\label{subsec:ModelProblems}
We are interested in flows on surfaces. As two prototypical examples, we consider the surface Stokes and surface Navier--Stokes equations. 

\paragraph{Surface Stokes} Given $\bm f: M \to \TM$, find $(\u,p) : M \to \TM \times \mathbb{R}$ such that
\begin{equation}\label{eq:surfaceStokes}
    \begin{cases}
    -  \mu\, \mathbf{div}(\underline{\bm{\epsilon}}(\u)) + \grad p \ =\ \bm f,\\
    \div \u \ =\ 0,\\
    \end{cases}\qquad\text{on }M
\end{equation}
possibly complemented by boundary conditions if $\Gamma \neq \emptyset$. Here, $\mu>0$ is the kinematic viscosity, $\underline{\bm{\epsilon}}(\u) = \frac12 ( \underline{\nabla} \u + \underline{\nabla} \u^T)$ is the linear strain tensor, and $\mathbf{div}(\cdot)$ is the surface divergence of tensor fields mapped into $\TM$.

\paragraph{Unsteady surface Navier--Stokes} Given $\u_0: M \to \TM$ and $\bm f:M \times (0,\tend]\to \TM$, find $(\u,p): M \times (0,\tend] \to \TM \times \mathbb{R}$ such that
\begin{equation}\label{eq:surfaceNavierStokes}
    \begin{cases}
    \partial_t \u + \dirder{\u} \u 
    - \mu\, \mathbf{div}(\underline{\bm{\epsilon}}(\u)) + \grad p \ =\ \bm f,\\
    \div \u \ =\ 0,
    \end{cases}
    \qquad \text{on }M\times (0,\tend].
\end{equation}
with initial condition $\u(\cdot,0) \ =\ \u_0$ and possible boundary conditions if $\Gamma \neq \emptyset$. Note that for $\mu=0$ we obtain the incompressible Euler equations.
 \section{Helmholtz--Hodge decomposition of divergence-free fields}\label{sec:harmonic}

This section develops the theoretical framework for the Helmholtz--Hodge decomposition of divergence-free vector fields on smooth and piecewise-smooth surfaces. We first recall the smooth Helmholtz--Hodge decomposition and its interpretation as a Hilbert complex. We then introduce a discrete analogue based on the Brezzi--Douglas--Marini ($\BDM$) finite element complex. 
Finally, we discuss how to characterize the orthogonal complement of divergence-free functions as a gradient space; and we deal with the lack of $\mathbf{H}^1$ conformity.

\subsection{Differential chain complexes}\label{sec:VelocityComplexes}

The classical $L^2$ de Rham chain complex on a smooth manifold $M$ is given by
\begin{equation}\label{eq:L2complex}
    \begin{tikzcd}
        L^2(M) \arrow[r,"\rot"] & \mathbf{L}^2(M) \arrow[r,"\div"] & L^2(M),
    \end{tikzcd}
\end{equation}
where the involved (unbounded) differential operators are densely defined, closed, and have closed range. In order to obtain bounded operators, we restrict these operators to their respective domains, yielding the associated \emph{domain complex} \cite{A18FEEC} of the $L^2$ de Rham complex
\begin{equation}\label{eq:Domaincomplex}
    \begin{tikzcd}[row sep=-0.3em]
        H^1(M) \arrow[r,"\rot"] & \mathbf{H}(\div,M) \arrow[r,"\div"] & L^2(M). \\
        \scalebox{0.75}{$\subset L^2(M)$} &  \scalebox{0.75}{$\subset \mathbf{L}^2(M)$} &
    \end{tikzcd}
\end{equation}
The Helmholtz--Hodge decomposition admits several variants depending on the chosen spaces. 
To derive the specific decomposition considered below in \Cref{eq:HkDecomp}, we augment \eqref{eq:Domaincomplex} to account for boundary conditions and additional regularity, obtaining
\begin{equation}\label{eq:H2complex}
    \begin{tikzcd}[row sep=-0.3em]
        H^2(M)\cap H^1_0(M) \arrow[r,"\rot"] & \mathbf{H}^1(M)\cap \mathbf{H}_0(\div,M) \arrow[r,"\div"] & L^2_0(M).
        \\
        \scalebox{0.75}{$\subset L^2(M)$} &  \scalebox{0.75}{$\subset \mathbf{L}^2(M)$} &
    \end{tikzcd}
\end{equation}

The sequences \eqref{eq:L2complex}, \eqref{eq:Domaincomplex}, and \eqref{eq:H2complex} are \textit{chain complexes}, i.e., 
$\div\circ \rot = 0$, and thus $\ran(\rot) \subset \ker(\div)$.
From a broader perspective, we can consider these chain complexes within the functional analytic theory of Hilbert chain complexes \cite{AFW06,AFW10,A18FEEC}, cf.~also \ref{sec:HilbertComplexes}. 

The structure of the $L^2$-chain complex \eqref{eq:L2complex} induces a Helmholtz--Hodge decomposition, which carries over to the chain complex \eqref{eq:H2complex}.
Provided that $M$ possesses a $C^{1,1}$ boundary, one obtains the following unique $L^2$-orthogonal decomposition into the $\grad$ of a \textit{potential space}, the $\rot$ of a \textit{streamfunction space}, and the space of \textit{harmonic fields}: 
\begin{align}\label{eq:HkDecomp}
    \mathbf{H}^1(M)&=\grad(H^{2}(M)\cap L^2_0(M)) \osum \rot(H^{2}(M)\cap H^1_0(M)) \osum \Hn, 
\end{align}
where the space of \textit{harmonic fields} defined as 
\begin{align}\label{eq:harmonic}
    \Hn \coloneqq \{\v\in \mathbf{J}_{L^2}(M)\cap \mathbf{H}(\mathrm{curl},M) \mid \mathrm{curl}(\v)=0\}
\end{align}
is finite-dimensional, and its dimension is equal to the first Betti number $b_1(M)$. The proof of \eqref{eq:HkDecomp} is given in \cite[Lem.~3.10]{StreamVortForm}, where it was also shown that $\Hn \hookrightarrow \mathbf{H}^1(M)$, cf. \cite[Sec. 3.2]{StreamVortForm}. \Cref{fig:splitting0} illustrates this decomposition. 


\begin{figure}[!htbp]
    \centering
    \vspace*{-0.2cm}
    \includegraphics[width=0.99\textwidth]{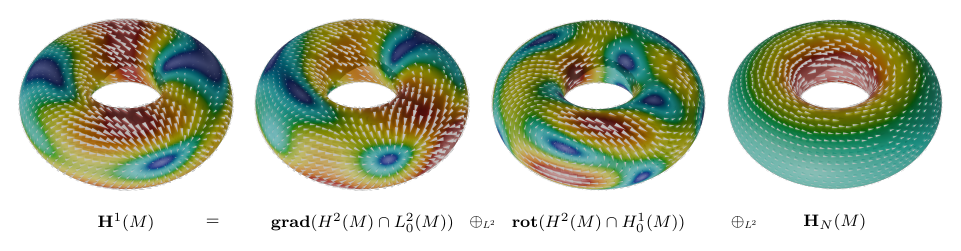}
    \vspace*{-0.6cm}
    \caption{Decomposition \eqref{eq:HkDecomp} of an $\mathbf{H}^1$ vector field into a gradient of a potential, rotation of a streamfunction, and a harmonic field (from left to right) on a torus.} 
    \label{fig:splitting0}
\end{figure}


For the divergence-free subspace, the gradient contributions vanish and the decomposition \eqref{eq:HkDecomp} yields
\begin{align}\label{eq:HkDecomp2}
    \mathbf{H}^1(M)\cap \mathbf{J}_{L^2}(M) &= \rot(H^{2}(M)\cap H^1_0(M)) \osum \Hn.
\end{align}
%

%

Our goal is to transfer this structure to the discrete surface $M_h$.
Theoretically, it is possible to define a corresponding $H^2(M_h)$ complex, but constructing finite element subspaces for both $H^2(M_h)$ and $\mathbf{H}^1(M_h)$ is practically infeasible.
%
To circumvent these difficulties, we return to the domain complex \eqref{eq:Domaincomplex} with appropriate boundary conditions. 
This reduces the regularity assumptions on the surface $M$, allowing us to work directly with piecewise smooth surfaces and to consider the chain complex directly on the discrete surface $M_h$:
\begin{equation}\label{eq:H1MhComplex}
    \begin{tikzcd}
        H^1_0(M) \arrow[r,"\rot"] \arrow[d,"\iota_h^1"] & \mathbf{H}_0(\div,M) \arrow[r,"\div"] \arrow[d,"\iota_h^2"] & L^2_0(M) \arrow[d,"\iota_h^3"]\\
        H^1_0(M_h) \arrow[r,"\rot"] & \mathbf{H}_0(\div,M_h) \arrow[r,"\div"] & L^2_0(M_h)
    \end{tikzcd}
\end{equation}
Here, the maps $\iota_h^i$ denote quasi-isomorphisms between the corresponding spaces.
The discrete surface $M_h$ provides enough regularity to define $\mathbf{H}^1$-vector fields, which suffices to make sense of the model problems considered in \Cref{subsec:ModelProblems}. However, discretizations based on this chain complex produce discrete velocities that are only $\mathbf{H}(\div,M_h)$-conforming, which may be insufficient for problems with continuous solutions in $\mathbf{H}^1$. Discretization strategies to address this issue are discussed in \Cref{ssec:nonconformity}.

The chain complex \eqref{eq:H1MhComplex} induces the following unique $L^2$-orthogonal decomposition for $M$
\begin{align*}
    \mathbf{L}^2=\grad(H^1(M)\cap L^2_0(M))\osum \rot(H^1_0(M))\osum \Hn,
\end{align*}
which for the velocity space implies that $\mathbf{J}_{L^2}=\rot(H^1_0(M))\osum \Hn$.
To approximate the chain complex, we seek a finite element sub-complex of the form
\begin{equation}\label{eq:H1complex} 
    \centering
    \begin{tikzcd}
        \hphantom{(\theequation a)}&& H^1_0(M_h) \arrow[r,"\rot"] \arrow[d,"\pi_h^1"] & \mathbf{H}_0(\div,M_h) \arrow[r,"\div"] \arrow[d,"\pi_h^2"] & L^2_0(M_h) \arrow[d,"\pi_h^3"] && (\theequation \mathrm{a}) \\
        \hphantom{(\theequation b)}&& \mathbb{S}\arrow[r,"\rot"] & \mathds{V} \arrow[r,"\div"] & \mathbb{Q} && (\theequation \mathrm{b})
    \end{tikzcd}
\end{equation}
where $\pi_h^i$ are suitable --- in the sense of Hilbert complexes, cf. \ref{sec:HilbertComplexes} --- bounded projection operators that turn \eqref{eq:H1complex} into a commutative diagram, and where $\mathbb{S}\subset H^1_0(M_h)$, $\mathds{V}\subset \mathbf{H}_0(\div,M_h)$, $\mathbb{Q}\subset L^2_0(M_h)$. 

The resulting discrete chain complex preserves the cohomology of the continuous chain complex, which is one of the central principles in finite element exterior calculus; see, for example, \cite{BL92,AFW06,AFW10,HS12,A18FEEC,H25} and \ref{sec:HilbertComplexes}. 
Most importantly, 
the space of discrete harmonic fields $\mathds{H}$, defined as the $L^2(M_h)$-orthogonal complement of $\rot( \mathbb{S})$ in $\mathds{J}\coloneqq \mathds{V}\cap\mathrm{ker}(\div)$, has the correct dimension, i.e., $\mathrm{dim}(\mathds{H})=b_1(M)$.
In other words, the discrete sub-complex correctly captures the topological structure of the harmonic fields. 





\begin{rem}[Mixed boundary conditions]
    The framework presented here extends naturally to the case of mixed boundary conditions, where the no-penetration constraint $\mathbf{u}\cdot\bm{\nu}_{\Gamma}=0$ 
    is imposed only on a subset of the boundary.
    The existence of a corresponding Helmholtz--Hodge decomposition, together with the characterization of the appropriate streamfunction, potential, and harmonic subspaces in this setting, is established in \cite[Thm 4.3]{gol2011hodge}.
\end{rem}

\subsection{$\BDM$ realization of the Helmholtz--Hodge decomposition}
\label{sec:bdmrealization}

In order to realize a compatible finite element sub-complex, we can consider the $\BDM$ finite element space \cite{BBF13}, which we define for $k\geq 0$ as 
\begin{align}
    \BDM^k \coloneqq \mathds{P}^k \cap \mathbf{H}(\div,M_h), \qquad \BDM_{0}^k \coloneqq \{\mathbf{v}_h\in \BDM^k \mid \mathbf{v}_h\cdot\bm{\nu}_{\Gamma_h}=0\}.
\end{align}
This space consists of (mapped) vectorial tangential polynomials that are normal-continuous in the sense that the normal jump $\jump{\cdot}_{\nu}$ vanishes. 
While the $\BDM$ space is typically defined for $k\geq 1$, we also allow for the lowest-order case $k=0$. 
To complement this vector-valued space, we define the scalar space $\mathbb{S}^{k+1}_0 \coloneqq \mathbb{P}^{k+1} \cap H^1_0(M_h)$ as the space of continuous Lagrange finite elements of degree $k+1$ with vanishing boundary traces or, in the case of $\Gamma=\emptyset$, zero integral mean. 
Letting $\Phi_T$ denote the mapping from the reference element $\hat{T}$ to $T \in \mathcal{T}_h$ and $\phi_h \in \mathbb{P}^k(T)$, cf. \eqref{eq:polynomialSpaces}, elementary calculations yield the relation
$\rot( \phi_h ) = \mathcal{J}_T^{-1} (D\Phi_T  \hat{\rot}(\hat{\phi}_h)) \circ \Phi_T^{-1}$ 
with 
$\hat{\phi}_h = \phi_h \circ \Phi_T \in \hat{\mathbb{P}}^k(\hat{T})$. 
I.e., the $\rot$ of polynomials transforms using the exact same Piola transformation as in the definition of $\mathds{P}^k(T)$, cf. \eqref{eq:piola}. Adding continuity we find $\rot \mathbb{S}^{k+1}_0 \subset \BDM^k_0$, even on curved surface triangulations. 
Further, we recall $\div \BDM^k_0 \subset \mathbb{P}_{\!\mathcal{J}}^{k-1}$, cf. \eqref{eq:pressurep}.

The choices $\mathbb{S} = \mathbb{S}^{k+1}_0$, $\mathds{V} = \BDM^k_0$, and $\mathbb{Q} = \mathbb{P}_{\!\mathcal{J}}^{k-1}$ in (\ref{eq:H1complex}b) give rise to a discrete chain complex.
As above, we define the spaces of discretely divergence-free velocities and discrete harmonic fields by
\begin{align}
    \JBDM \coloneqq \BDM^k_0 \cap \ker(\div), \qquad 
    \HBDM \coloneqq \bigl(\rot(\mathbb{S}^{k+1}_0)\bigr)^\perp \subset \JBDM,
\end{align}
respectively, where orthogonality is understood with respect to the $L^2(M_h)$ inner product in $\JBDM$.

Since the Raviart--Thomas space $\mathds{RT}_0^k$ differs from $\BDM_0^k$ only by the inclusion of internal, non-divergence-free bubble functions \cite{BBF13}, the corresponding divergence-free subspaces coincide, i.e.,
\[
\mathds{RT}_0^k \cap \ker(\div) = \JBDM, \qquad k \geq 0.
\]
Consequently, both spaces share the same cohomological structure, which justifies our restricting attention to the $\BDM$ complex.

In order to show that the finite element sub-complex correctly preserves cohomological structure, we could construct suitable cochain projections $\pi_h^i$ and apply the framework of Hilbert complexes. However, these constructions are usually delicate, because one needs to circumvent the issue that standard interpolation operators are only defined with additional regularity assumptions; see, for instance \cite{BBF13}. Projections with only minimal regularity requirements have been developed for the flat case in \cite{Christiansen07,S01} and only recently for surfaces \cite{licht2023}. 

Instead, we use a straightforward dimension-counting argument (cf.~\ref{thm:countBDM}) to obtain $\dim(\HBDM) = b_1(M) = \dim(\Hn)$. As a consequence, our approach correctly captures first cohomology, which is the only non-trivial one for surfaces. 

We summarize these observations in the following result.

\begin{thm}[Decomposition of $\JBDM$ and $\BDM^k_0$]
    \label{thm:HodgeDecompJBDM}
    One has the following discrete $L^2(M_h)$-orthogonal decomposition: 
    \begin{subequations}
    \begin{align}
        \JBDM &= \mathbf{rot}(\mathbb{S}^{k+1}_0) \osum \HBDM.
        \label{eq:JBDMDecomp}
        \intertext{
    The discrete harmonic fields capture the correct topology, as $\dim(\HBDM) = b_1(M)$.
    As a direct consequence, we also have the decomposition }
        \BDM^k_0 &= \mathbf{rot}(\mathbb{S}^{k+1}_0) \osum \HBDM \osum (\JBDM)^{\perp}. \label{eq:BDMDecomp}
    \end{align}
    \end{subequations}
\end{thm}

\begin{figure}[!htbp]
    \centering
    \includegraphics[width=0.99\textwidth]{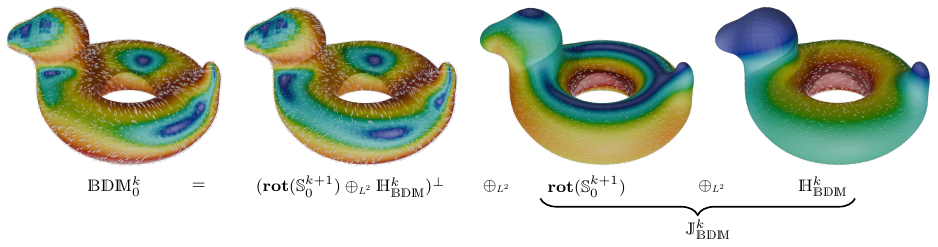}
    \caption{Visualization of the discrete decomposition \eqref{eq:BDMDecomp} of a vector field in $\BDM^k_0$ into a harmonic and rotation field and remaining complement (from right to left) on the ``bob'' geometry from \cite{crane2013robust}. Computed through the procedure explained in \Cref{sec:harmonic:algo} below.}
    \label{fig:splitting1}
\end{figure}

This theorem implies that the discretization of a divergence-free vector field in $\JBDM$ is equivalent to a discretization of a streamfunction in $\mathbb{S}^{k+1}_0$ and a harmonic field in $\HBDM$. 
A global basis for the space of discrete harmonic fields $\HBDM$ can be constructed through the algorithm described in \Cref{sec:harmonic:algo}.
Notably, the decomposition in \eqref{eq:JBDMDecomp} enables the implementation of higher-order variants of the methods in \cite{YNWWC23,ZhuYinChern2025,StreamVortForm}, which have thus far relied on the lowest-order ($k=0$) splitting.

\begin{rem}[Hierarchy of the $\BDM_0^k$ complex]
    \label{rem:hierarchy}
For $k\geq1$, 
we can decompose the $\BDM_0^k$ space into the lowest order component $\mathds{RT}_0^0$ and a higher-order complement $(\mathds{RT}_0^0)^c$.
The divergence-free subspace of $(\mathds{RT}_0^0)^c$ can be written as $\rot$ of $(\mathbb{S}^{1}_0)^c$, which is the complement of $\mathbb{S}^{1}_0$ in $\mathbb{S}^{k+1}_0$; see, for example, \cite{schoberl2005high,zaglmayr2006high}. 

Exploiting the Helmholtz--Hodge decomposition \eqref{eq:JBDMDecomp} of $\mathds{J}_{\BDM}^0$, we find
\begin{equation}\label{eq:jbdmh0}
\JBDM = \mathds{J}_{\BDM}^0 + (\mathds{J}_{\BDM}^0)^c 
\stackrel{\eqref{eq:JBDMDecomp}}{=} \rot(\mathbb{S}^{1}_0) + \mathds{H}_{\BDM}^0 + \rot((\mathbb{S}^{1}_0)^c)
= \rot(\mathbb{S}^{k+1}_0) + \mathds{H}_{\BDM}^0.
\end{equation}
This reveals that no higher order moments of harmonic fields are required to correctly capture topology; the lowest-order component $\mathds{J}_{\BDM}^0$ already contains the full harmonic structure.
\end{rem}


\begin{rem}[$H^2$-complex in the planar case]\label{rem:H2complex}
    For the special case of planar polygonal domains $M_h = M \subset \mathbb{R}^2$, it is feasible to directly discretize the $H^2$-complex from \eqref{eq:H2complex}, but the $H^2$-conforming stream function and the $H^1$-conforming velocity spaces have to be chosen carefully. For example, as explained in \ref{sec:SVcomplex}, the $H^2$-conforming Argyris element is not compatible with the $H^1$-conforming Scott--Vogelius element. Rather, one has to choose the Morgan-Scott space \cite{MS75} as the stream function space or modify the velocity space \cite{FN13}. 
    Another approach is to employ a Clough--Tocher (or Alfeld) split, allowing us to use the Hsieh--Clough--Tocher ($\mathbb{HCT}^{k+1}_0$), $k \ge 2$, element to replace the Argyris element and obtain the following chain complex
        \begin{equation}
    \centering
    \begin{tikzcd}
        H^2(M)\cap H^1_0(M) \arrow[r,"\rot"] \arrow[d,"\pi_h^1"] & \mathbf{H}^1(M)\cap\mathbf{H}_0(\div,M)\arrow[r,"\div"] \arrow[d,"\pi_h^2"] & L^2_0(M) \arrow[d,"\pi_h^3"] \\
        \mathds{HCT}_0^{k+1} \arrow[r,"\rot"] & \mathds{SV}_0^k \arrow[r,"\div"] & \mathbb{P}^{k-1}\cap L^2_0(M) 
    \end{tikzcd}
            \end{equation}
    This construction correctly captures the dimension of the harmonic fields, as shown in \Cref{thm:countHCT} or in \cite[Sec.~4.3]{JLMNR17} for the simply connected case.
    For a detailed discussion, we also refer to \cite{H25}.
\end{rem}

For the discretization of incompressible flows, such as the model problems from \Cref{subsec:ModelProblems}, it is sufficient to consider the characterization of the divergence-free subspace in \eqref{eq:JBDMDecomp}. In general, however, it is also interesting to consider the orthogonal complement $(\mathbb{J}_{\mathbb{BDM}}^k)^{\perp}$ to represent \emph{non-divergence-free} vector fields. To provide a mathematical exact characterization of this space, we may consider the discrete $L^2$-adjoint $\addiv: \mathbb{P}^{k-1} \to \mathbb{BDM}^k$, defined such that $(\addiv p, \mathbf{v})_{L^2(M_h)} = -(p, \div \mathbf{v})_{L^2(M_h)}$ for all $\mathbf{v} \in \mathbb{BDM}^k$.
Using this operator, we obtain the exact decompositions of the $\BDM^k$ spaces:
\begin{subequations} \label{eq:BDMDecomp2}
\begin{align}
    \mathds{BDM}^k_0 &= \rot(\mathbb{S}^{k+1}_0)\osum \HBDM \osum \addiv(\mathbb{P}^{k-1}\cap L^2_0(M_h)), \\
    \mathds{BDM}^k &= \rot(\mathbb{S}^{k+1}_0)\osum \rot(\mathbb{S}^{k+1}_{\Gamma_h})\osum \HBDM \osum \addiv(\mathbb{P}^{k-1}),
\end{align}
\end{subequations}
where $\mathbb{S}^{k+1}_{\Gamma_h}$ contains the elements of $\mathbb{S}^{k+1}$ associated to the boundary degrees of freedom. For the pressure reconstruction discussed in \Cref{sec:variants:press}, it is convenient to define $\addiv$ in  $\mathbb{P}^{k-1}$, but we could analogously define it in $\mathbb{P}_{\mathcal{J}}^{k-1}$. This leads to a slightly different $\addiv$ operator. 

\subsection{Incomplete discrete Helmholtz--Hodge decompositions} \label{ssec:searchgrad}
The construction of $\addiv$ discussed above is done implicitly via the discrete $L^2$-adjoint rather than through an explicit local construction.
Ideally, a complete discrete Helmholtz--Hodge decomposition would characterize this complement, both within $\BDM^k_0$ and the full discontinuous polynomial space $\mathds{P}^k$, as a pure discrete gradient space.
However, identifying standard scalar finite element spaces that generate these exact gradient complements is a highly non-trivial task. For instance, even in the simplified setting of simply connected planar domains with straight conforming triangulations $M_h = M \subset \mathbb{R}^2$ where the harmonic fields vanish, it has been shown that exact decompositions for $\mathds{P}^k$ using standard finite element spaces are severely limited \cite{BringmannKettelerSchedensack2023}.
%

Specifically, as shown in \cite[Thm 3.3]{BringmannKettelerSchedensack2023}, a complete decomposition is only straightforward for $k=0$, where the full space decomposes exact into
\begin{equation*}
    \mathds{P}^0 = \rot(\mathbb{S}^1_0) \osum \grad_h(\mathbb{CR}^1), \qquad \mathbb{CR}^1=\{v\in \mathbb{P}^1 \mid\Pi_{\mathbb{P}^0(\Eh)}\jump{v}=0\}\cap L^2_0(M_h),
\end{equation*}
where $\grad_h$ denotes the element-wise (broken) gradient and $\mathbb{CR}^1$ is the non-conforming Crouzeix--Raviart space; see also~\cite{WardetzkyThesis}. For $k=1$, a decomposition is only possible using two non-conforming spaces, and it is strictly impossible with any conforming space, such as $\mathbb{S}^2_0$. 
For polynomial orders $k \geq 2$, the fundamental obstruction is dimensional: the sum of standard compatible potential and streamfunction spaces simply lacks the degrees of freedom required to span the entirety of the discontinuous polynomial space \cite[Rem 5.4]{BringmannKettelerSchedensack2023}. 
Similarly, we do not expect to characterize $(\JBDM)^{\perp}$ purely as a standard gradient space.

However, if we accept the Helmholtz--Hodge decomposition to be \emph{incomplete}, we can still identify a suitable gradient subspace of $(\JBDM)^{\perp}$ to mirror the continuous decomposition \eqref{eq:HkDecomp}, explicitly capturing the structure of a rotational, a harmonic, and a gradient component.
For the following discussion, we assume that the surface $M_h$ consists of linear triangles to avoid issues stemming from the element-transformations, cf.~\Cref{rem:curved}. Assuming that $k+1$ is odd, we define the Crouzeix--Falk space with zero mean as $\mathbb{CF}^{k+1} \coloneqq \{v\in \mathbb{P}^{k+1} \mid\Pi_{\mathbb{P}^k(\mathcal{E}_h)}\jump{v}=0\} \cap L^2_0(M_h)$. This space is a canonical choice, since the normal trace conditions match those of $\mathbb{BDM}^k$ which ensures the required $L^2$-orthogonality between the broken gradients and the divergence-free fields. Then, we obtain
\begin{equation}\label{eq:HodgeNotPk}
    \rot(\mathbb{S}^{k+1}_0) \osum \HBDM \osum \grad_h(\mathbb{CF}^{k+1}) \subset \mathds{P}^k.
\end{equation}
This choice for the gradient space is not unique, and we could also choose, for example, the space $\grad(\mathbb{S}^{k+1}\cap L^2_0(M_h))$. 
Even though the decomposition \eqref{eq:HodgeNotPk} is not complete, we could still use it to discretize the components $\psi$, $\bm{h}$, $\phi$ of the Helmholtz--Hodge decomposition \eqref{eq:HkDecomp} of vector field $\v \in \mathbf{H}^1(M)$.

\begin{rem}[Curved triangulations]\label{rem:curved}
To transfer results of the (incomplete) discrete Helmholtz--Hodge decomposition from linear to curved triangulations, one requires suitable transformations of the discrete vector-valued and scalar function spaces. While such transformations can be chosen so that they preserve the structure of two components, for instance the streamfunction and velocity spaces, cf.~\Cref{sec:bdmrealization}, there exists no choice of transformations that simultaneously preserves compatibility of all components of the discrete Helmholtz--Hodge decomposition.
\end{rem}

\subsection{Handling the lack of $H^1$-conformity} \label{ssec:nonconformity} 

As discussed in \Cref{sec:VelocityComplexes}, we advocate using the $H^1\to\mathbf{H}(\div)\to L^2$ chain complex and its discretization rather than the $H^2\to \mathbf{H}^1\to L^2$ chain complex.
Besides the difficulty of constructing $\mathbf{H}^1$-conforming finite elements on surfaces that possess only $C^{0,1}$-regularity, 
the principal motivation is that with the $\BDM$ velocity space the former sequence admits a clean, cohomology-preserving discrete realization, which is essential for the structure-preserving Helmholtz--Hodge decomposition developed above (cf.\ \Cref{thm:HodgeDecompJBDM}).

An important consequence of this choice is that discrete velocity fields are only $\mathbf{H}(\div)$-conforming and \emph{not} $\mathbf{H}^1$-conforming. For problems whose well-posed weak formulation usually demand $\mathbf{H}^1$-regularity of the velocity, for example the model problems from \Cref{subsec:ModelProblems}, the resulting discretization is therefore \emph{non-conforming} with respect to the velocity space. There are two principal strategies to address this mismatch:

\begin{enumerate}[(i)]
    \item \textbf{Reformulation as a mixed system.}
    The problem can be reformulated so that the velocity in the well-posed weak formulation is required to lie only in $\mathbf{H}(\div)$ or less by introducing additional unknowns that absorb the higher-order differential operators.
    Let us mention two examples. 
    In the streamfunction--harmonic--vorticity formulation of \cite{YNWWC23,ZhuYinChern2025,StreamVortForm}, for example, introducing the vorticity as a separate unknown yields a formulation in which second-order operators act only on the \emph{scalar} streamfunction and \emph{scalar} vorticity; the velocity no longer appears as a primary unknown but is recovered as a derived quantity in $\mathbf{H}(\div)$.
    From a different angle, the mass-conserving mixed stress (MCS) formulation of Gopalakrishnan, Lederer, and Sch\"oberl \cite{GopalakrishnanLedererSchoeberl2020} introduces the deviatoric stress tensor as a new unknown.
    This splits the symmetric gradient operator and yields a saddle-point formulation in which the velocity is treated conformingly in the sense that it naturally belongs to $\mathbf{H}(\div)$.
    While the MCS formulation has so far been developed for flat domains, an extension to surfaces would be a natural complement to the present framework.

    \item \textbf{Non-conforming discretization via DG or Hybrid DG methods.}
    The alternative, which we pursue in this work (cf.\ \Cref{sec:ref-hdiv-dg} and below), is to accept the non-conformity and directly discretize the PDE in its original (second-order) form using discontinuous or hybrid discontinuous Galerkin (DG/HDG) methods.
    In this approach, discrete solutions are not $\mathbf{H}^1$-conforming; instead, tangential continuity across element interfaces is enforced \emph{weakly} --- i.e.\ in the limit of mesh refinement --- through the discrete formulation, for instance via consistent (hybrid) symmetric interior penalty terms \cite{LedererLehrenfeldSchoeberl2020,LehrenfeldSchoeberl2016}.
\end{enumerate} 
\section{Discretizations in the divergence-free subspace of $\BDM$ for surface flows}\label{sec:ref-hdiv-dg}
This section describes how the divergence-free subspace $\JBDM$ of the BDM complex can be used for the numerical solution of surface flow problems. First, we outline a randomized algorithm for constructing a basis of the discrete harmonic fields $\HBDM$ and then we formulate the discrete problem directly in the
divergence-free space using the streamfunction--harmonic split. Finally, we discuss non-conforming discretizations for Surface Stokes and Navier--Stokes equations.

Let us note that it suffices to compute $\mathds{H}_{\mathds{BDM}}^0$ to form $\JBDM$, cf. \eqref{eq:jbdmh0} in \Cref{rem:hierarchy}. 
Computing $\mathds{H}_{\mathds{BDM}}^k$ is only relevant to conveniently represent the harmonic part of the solution.

\subsection{Randomized construction of the harmonic basis}\label{sec:harmonic:algo}
The space of discrete harmonic fields $\mathds{H}_{\mathds{BDM}}^k$ is defined globally 
and does not admit a local basis. However, we can construct a global basis by adapting the randomized strategy of \cite[Algo 4]{YNWWC23}, which yields an $L^2(M_h)$-orthonormal basis of $\mathds{H}_{\mathds{BDM}}^k$ with probability one in $\dim(\HBDM)$ steps. This construction of discrete harmonic fields can be generalized to the setting of abstract Hilbert chain complexes, cf. \ref{ssec:RandomizedHarmonicAbstractHilbertComplexes}.

\begin{algo}[Harmonic basis via divergence-free projection]\label{alg:harmonic}
Let $n=b_1(M) = \dim(\HBDM)$ be the number of harmonic fields required. \\
Initialize $i=1$ and $B=\emptyset$.
\begin{enumerate}[1.)]
    \item \textbf{Random divergence-free sample.} 
   With respect to some (arbitrary) regular distribution, choose a random vector $\mathbf{r}_h \in \mathds{BDM}^k$ (or $\mathds{P}^k$),  $\|\mathbf{r}_h\|_{L^2(M_h)}=1$, 
    and compute its $L^2(M_h)$-projection into $\mathds{J}^k_{\mathds{BDM}}$
    \[
    \mathbf{u}_h \coloneqq  \Pi_{\JBDM} \bm r_h \ =\ \underset{\bm u_h\in \mathds{J}^k_{\mathds{BDM}}}{\operatorname*{argmin}}\, \norm{\bm u_h - \bm r_h}_{L^2}.
    \]
    In \Cref{ssec:helmholtzproj} we discuss efficient realizations of the projection $\Pi_{\JBDM}$.
    \item \textbf{Remove streamfunction part.}
    Compute the streamfunction $\psi_h \in \mathbb{S}^{k+1}_0$ by solving the Laplace problem
    \begin{equation*}
        (\rot(\psi_h), \rot(\phi_h))_{M_h} = {\color{gray!80} (\grad(\psi_h), \grad(\phi_h))_{M_h} =} \, (\mathbf{u}_h, \rot(\phi_h))_{M_h},~\forall \phi_h \in \mathbb{S}^{k+1}_0.
    \end{equation*}
    Define the candidate harmonic field as $\mathbf{w}_h \coloneqq \mathbf{u}_h - \rot(\psi_h) \perp \rot(\mathbb{S}^{k+1}_0)$.
    
    \item \textbf{Orthogonalization and Update.} 
    Orthogonalize $\mathbf{w}_h$ against the 
    basis vectors in $B$ 
    with respect to the $L^2(M_h)$-inner product yielding $\tilde{\mathbf{w}}_h$. 
    If $\|\tilde{\mathbf{w}}_h\|_{L^2(M_h)}<\texttt{tol}$, discard the sample and return to Step 1.).
    Otherwise, normalize the result to obtain $\mathbf{h}^i_h$. 
    Accept the new basis vector by setting $B \leftarrow B\cup\{\mathbf{h}_h^i\}$, increment $i \leftarrow i+1$, and repeat until $i > n$.
\end{enumerate}
The resulting set $B=\{\mathbf{h}_h^i\}_{i=1}^n$ forms an $L^2(M_h)$-orthonormal basis of $\mathds{H}_{\mathds{BDM}}^k$. 
\end{algo}





\subsection{Discretizing PDEs in $\JBDM$} \label{ssec:discinJBDM}
We want to solve a discrete problem directly in the divergence-free subspace $\JBDM$ to discretize a problem with a divergence free solution. Thus, we want to find $\bm{u}_h \in \JBDM$ such that
\begin{equation}\label{eq:genericproblem}
    \mathcal{F}_h(\u_h,\v_h) = f_h(\v_h) \qquad \forall\, \bm{v}_h \in \JBDM,
\end{equation}
where $\mathcal{F}_h(\cdot,\cdot)$ is a semilinear form representing the discretization of the respective model problem at hand. 
In \cite{LedererLehrenfeldSchoeberl2020}, the problem is rather formulated as a saddle-point problem in the full $\BDM$ space, where the divergence-free constraint is enforced via a Lagrange multiplier (the pressure):
find $(\u_h, p_h) \in \BDM^k_0 \times \mathbb{P}^{k-1}$ such that
\begin{subequations}\label{eq:genericproblem:saddlepoint}
    \begin{align}
        \hspace*{0.1\textwidth}\mathcal{F}_h(\u_h,\v_h) + (\div \v_h, p_h)_{M_h} &= f_h(\v_h) && \forall\, \bm{v}_h \in \BDM^k_0, \hspace*{0.1\textwidth}~\\
        \hspace*{0.1\textwidth}(\div \u_h, q_h)_{M_h} &= 0 && \forall\, q_h \in \mathbb{P}^{k-1}. \hspace*{0.1\textwidth}~
    \end{align}
\end{subequations}
The velocity solution $\u_h$ of \eqref{eq:genericproblem:saddlepoint} is the \emph{same as the solution of \eqref{eq:genericproblem}}; in particular, $\u_h \in \JBDM$, cf. \cite[Lemma 1]{LedererLehrenfeldSchoeberl2020}.
In the following however, we want to avoid the saddle-point formulation and directly solve \eqref{eq:genericproblem} in $\JBDM$. 

Formally, we introduce a Galerkin isomorphism between $\R^{N_{\mathbb{J}}}$ and $\JBDM$ with $N_{\mathds{J}} \coloneqq N_{\mathbb{S}}+N_{\mathds{H}}$, where $N_{\mathbb{S}}\coloneqq\dim(\mathbb{S}^{k+1}_0)$ is the dimension of the streamfunction space and $N_{\mathds{H}} \coloneqq \dim(\HBDM) =b_1(M)$ is the number of harmonic fields. With $\{\varphi_i\}_{i=1}^{N_{\mathbb{S}}}$ being a basis of $\mathbb{S}^{k+1}_0$ and $\{\bm{h}_h^j\}_{j=1}^{N_{\mathbb{H}}}$ being a basis of $\HBDM$, we define 
\begin{equation}\label{eq:Galisomorphism}
    \mathcal{G}_{\mathbb{J}}:\; \R^{N_{\mathbb{J}}} \;\to\; \BDM_0^k, \qquad
    \mathcal{G}_{\mathbb{J}} \bm{x} \;:=\; 
    \sum_{i=1}^{N_{\mathbb{S}}} x_{i} \,\rot(\varphi_i)
    \;+\; 
    \sum_{j=1}^{N_{\mathbb{H}}} x_{j+N_{\mathbb{S}}}\, \bm{h}_h^j.
\end{equation}
In particular, we can write any $\bm{u}_h \in \JBDM$ uniquely as
$\bm{u}_h = \rot(\psi_h) + \bm{h}_h$ for $\psi_h \in \mathbb{S}^{k+1}_0, \bm{h}_h \in \HBDM$.
We could now directly insert this decomposition into \eqref{eq:genericproblem} and formulate the problem directly in terms of the unknowns $(\psi_h, \bm{h}_h)$: find $(\psi_h, \bm{h}_h) \in \mathbb{S}^{k+1}_0 \times \HBDM$ such that
\begin{align}\label{eq:fourthOrder}
    \mathcal{F}_h(\rot(\psi_h)+\bm{h}_h, \rot(\phi_h)+\bm{k}_h) = f_h(\rot(\phi_h)+\bm{k}_h)
    \qquad \forall\, (\phi_h, \bm{k}_h) \in \mathbb{S}^{k+1}_0 \times \HBDM.
\end{align}
However, a realization of \eqref{eq:fourthOrder} requires additional attention due to the additional derivatives in the streamfunction part and the non-local nature of the harmonic basis functions if $\HBDM \neq \{ 0 \}$.

To circumvent these difficulties, we can exploit that $\JBDM$ is a subspace of a larger \emph{parent space} $\mathbf{V}_h \supset \JBDM$, for example, $\mathbf{V}_h = \mathds{BDM}^k_0$ or $\mathbf{V}_h = \mathds{P}^k$. Then, we can represent 
the Galerkin isomorphism $\mathcal{G}_{\mathbb{J}}$ from \eqref{eq:Galisomorphism} through a Galerkin isomorphism $\mathcal{G}_{\mathbf{V}}: \R^{N_{\mathbf{V}}} \to \mathbf{V}_h$ of the parent space $\mathbf{V}_h$ and a linear injection $T: \R^{N_\mathds{J}} \to \R^{N_{\mathbf{V}}}$ such that $\mathcal{G}_{\mathbb{J}} = \mathcal{G}_{\mathbf{V}} \circ T$.  
Instead of assembling the problem in $\JBDM$ directly, one sets up the discrete operator in the parent space $\mathbf{V}_h$, for which standard finite element infrastructure is available, and then restricts it to $\JBDM$ via the embedding $T$. Depending on the implementation framework, this could be done element-wise but also globally.

Both viewpoints lead to the same discrete problem.
The direct approach \eqref{eq:fourthOrder} sets up algebraic systems based on a basis in $\JBDM = \rot(\mathbb{S}^{k+1}_0) \osum \HBDM$, leading to the algebraic problem: Find $\bm{x} \in \R^{N_\mathds{J}}$ such that
\begin{equation} \label{eq:algsystem}
\mathcal{F}_h(\mathcal{G}_{\mathbb{J}}\bm{x},\, \mathcal{G}_{\mathbb{J}}\bm{y}) = f_h(\mathcal{G}_{\mathbb{J}}\bm{y}) \qquad  \forall \bm{y} \in \R^{N_\mathds{J}}.
\end{equation}
In the embedding approach, \eqref{eq:algsystem} is assembled by restricting the parent-space operator:
letting $\mathbf{A}_{\mathbf{V}}$ be the stiffness matrix of $\mathcal{F}_h(\mathcal{G}_{\mathbf{V}}\cdot,\,\mathcal{G}_{\mathbf{V}}\cdot)$ and $\bm{f}_{\mathbf{V}}$ the corresponding load vector, \eqref{eq:algsystem} is equivalent to the reduced system
\begin{equation}\label{eq:reducedsystem}
    T^\top \mathbf{A}_{\mathbf{V}}\, T\, \bm{x} = T^\top \bm{f}_{\mathbf{V}}.
\end{equation}
For the numerical experiments in \Cref{sec:numerics} we construct $T$ via the \emph{conforming (embedded) Trefftz} approach \cite{jcmeyermaster,LehrenfeldStockerTrefftz}, choosing $\mathbf{V}_h = \mathds{P}^k$ (discontinuous piecewise polynomials) as the parent velocity space.
The construction of $T$ proceeds element-locally for the streamfunction columns.
On each element $T \in \mathcal{T}_h$, one works in the discontinuous parent spaces $\mathds{P}^{k+1}(T) \times \mathds{P}^k(T)$ for the streamfunction and velocity respectively, and imposes two conditions: a \emph{Trefftz condition} enforcing $\mathbf{u}_h|_T = \rot(\phi_h)|_T$ in $L^2(T)$, and a \emph{conformity constraint} ensuring that $\phi_h$ agrees on shared boundaries with a globally continuous function in $\mathbb{S}^{k+1}_0$.
For each basis function $\varphi_i$ of $\mathbb{S}^{k+1}_0$ supported on $T$, this leads to a small element-local linear system
\begin{equation}\label{eq:TrefftzLocal}
    K_T\, T_T = L_T,
\end{equation}
where $K_T$ encodes both conditions and $L_T$ collects the right-hand sides from the conforming streamfunction basis.
The velocity rows of the solution $T_T$ provide the local columns of $T$ for the streamfunction degrees of freedom; these element-local solves are independent and can be performed in parallel.
The $N_{\mathds{H}}$ columns of $T$ corresponding to the harmonic basis vectors $\{\bm{h}^j_h\}_{j=1}^{N_{\mathds{H}}} \subset \mathds{P}^k$ are appended globally, completing the embedding.

\subsection{Discretization of the surface Stokes and surface Navier--Stokes equations}
To discretize the surface Stokes and Navier--Stokes problems \eqref{eq:surfaceStokes} and \eqref{eq:surfaceNavierStokes} with $\BDM^k_0$, respectively $\JBDM$, as the velocity space, we consider standard discontinuous Galerkin (DG) formulations for the viscous and advective terms following \cite{LedererLehrenfeldSchoeberl2020}. Recall the notation from \Cref{subsec:triangulation}, in particular \Cref{fig:sketchofedgepatch}. 

For $\u_h, \v_h \in  \mathds{BDM}^k_0$, we define the symmetric interior penalty bilinear form $a_h(\cdot,\cdot)$ as 
\begin{subequations}\label{eq:StokesSIP}
    \begin{align}
        a_h(\u_h,\v_h) \coloneqq &\sum_{T \in \Th}\int_T \mu\, \underline{\bm{\epsilon}}(\u_h) : \underline{\bm{\epsilon}}(\v_h) \, \mathrm{d}x + \sum_{E \in \Eh} \int_E \avg{- \mu\, \underline{\bm{\epsilon}}(\u_h)} \jump{\v_h}_{\bm{\tau}} \, \mathrm{d}s \\
        &+ \sum_{E \in \Eh} \int_E \avg{- \mu\, \underline{\bm{\epsilon}}(\v_h)} \jump{\u_h}_{\bm{\tau}} \, \mathrm{d}s + \frac{\alpha \mu}{h} \sum_{E \in \Eh} \int_E \jump{\u_h}_{\bm{\tau}} \jump{\v_h}_{\bm{\tau}} \, \mathrm{d}s
    \end{align}
\end{subequations}
Then, a DG discretization of \eqref{eq:surfaceStokes} on $\JBDM$ reads as: find $\u_h \in \JBDM$ such that 
\begin{equation}\label{eq:discreteSurfaceStokes}
    a_h(\u_h,\v_h) = f_h(\v_h) \quad \forall \v_h \in \JBDM.
\end{equation}
Introducing the pressure as a Lagrange multiplier to enforce the divergence-free constraint on a $\BDM^k_0$-velocity space, cf. \eqref{eq:genericproblem:saddlepoint}, yields the Stokes discretization from \cite[Sec. 3.3]{LedererLehrenfeldSchoeberl2020}.


\begin{remark}
    On simply connected flat domains $M \subset \mathbb{R}^2$, we have that $\HBDM = \{ 0 \}$ and may replace $\underline{\bm{\epsilon}}(\u_h)$ by $\grad(\u_h)$. In this setting, the fourth order problem \eqref{eq:fourthOrder} corresponds to a $C^0$-interior penalty discretization of the biharmonic equation with $\mathbb{S}^{k+1}_0$, as previously observed in \cite{KanschatSharma2014}: 
    \begin{subequations}
        \begin{align*}
            a_h(\rot(\psi_h), \rot(\phi_h)) = &\sum_{T \in \Th} \int_{T} \mu \nabla^2 \psi_h : \nabla^2 \phi_h \, \mathrm{d}x - \sum_{E \in \Eh} \int_{E} \avg{ \mu \nabla^2 \psi_h} \jump{\grad\phi_h} \, \mathrm{d}s \\
            - &\sum_{E \in \Eh} \int_{E} \avg{\mu \nabla^2 \phi_h} \jump{\grad \psi_h} \, \mathrm{d}s + \frac{\alpha \mu}{h} \sum_{E \in \Eh} \int_{E} \jump{\grad \psi_h} \jump{\grad \phi_h}\, \mathrm{d}s.
        \end{align*}
    \end{subequations}
\end{remark}
%
%
%
%
To discretize the surface Navier--Stokes equation \eqref{eq:surfaceNavierStokes} (in space), we additionally require a  discrete formulation of the convection term. 
We choose a standard upwind DG treatment of the convection term, cf. e.g. \cite[Sec.~3.4]{LedererLehrenfeldSchoeberl2020}.
For $\u_h, \v_h \in \BDM_0^k,~ \w_h \in \JBDM$, we define
\begin{align}
    c_h(\w_h; \u_h,\v_h) \coloneqq & \sum_{T \in \Th} \left( - \int_{T} \u_h \partial_{\w_h} \v_h \,\mathrm{d}x + \int_{\partial T} (\w_h \cdot \nv) (\u_h^{\text{up}} \cdot \v_h) \,\mathrm{d}s \right) 
    \label{eq:conv-form}
\end{align}
with the upwind choice defined as 
$
    \u_h^{\text{up}}(x) \coloneqq (\u_h \cdot \nv_{1}) \cdot \nv_{1} + 
    (\u_h(x) \vert_{T^\ast} \cdot \tv) \cdot \tv 
$
where $T^\ast$ is the upwind element, i.e. the element adjacent to $x\in\partial T$ with $\w_h \cdot \nv_{T^\ast} > 0$.
Crucially, this upwind choice guarantees convection robustness and energy stability, cf. \cite[Sec. 3.5.2]{LedererLehrenfeldSchoeberl2020}. 

Altogether, we consider a semi-discrete discretization of the unsteady surface Navier--Stokes equations \eqref{eq:surfaceNavierStokes} as follows: find $\u_h \in \JBDM$ such that
\begin{subequations} \label{eq:disc:SNSE}
    \begin{align}
        (\partial_t \u_h, \v_h)_{M_h} + a_h(\u_h,\v_h) + c_h(\u_h;\u_h,\v_h) &= (\bm{f}, \v_h)_{M_h} &&\forall \v_h \in \JBDM,~t \in (0,T], \label{eq:NSvh} \\
        (\u_h,\v_h)_{M_h}  &=  (\u_0,\v_h)_{M_h} &&\forall \v_h \in \JBDM,~t = 0.
    \end{align}
\end{subequations}

 \section{Implementation aspects}\label{sec:implementation}

We devote this section to practical aspects of realizing the discrete formulations described in \Cref{sec:ref-hdiv-dg}.
After describing how the dense parts of the system matrix corresponding to the non-local harmonic basis vectors can be handled efficiently, we briefly discuss hybridization as a means to reduce the computational cost of DG methods. We then address the efficient computation of the Helmholtz projection and the pressure reconstruction.

\subsection{Schur complement treatment of harmonic unknowns}\label{sec:impl:schur}

After discretizing \eqref{eq:genericproblem} in $\JBDM$ and possibly linearizing, we obtain a linear system of dimension $N_{\mathds{J}} = N_{\mathbb{S}} + N_{\mathds{H}}$.
Ordering the unknowns with the $N_{\mathbb{S}}$ streamfunction coefficients first and the $N_{\mathds{H}}$ harmonic coefficients second, the system matrix takes the block form
\begin{equation}\label{eq:blocksystem}
    \begin{bmatrix} A_{\mathbb{SS}} & A_{\mathbb{S}\mathds{H}} \\ A_{\mathds{H}\mathbb{S}} & A_{\mathds{HH}} \end{bmatrix}
    \begin{bmatrix} \bm{x}_{\mathbb{S}} \\ \bm{x}_{\mathds{H}} \end{bmatrix}
    =
    \begin{bmatrix} \bm{b}_{\mathbb{S}} \\ \bm{b}_{\mathds{H}} \end{bmatrix}.
\end{equation}
The streamfunction block $A_{\mathbb{SS}}$ is sparse (reflecting local finite element interactions), while the off-diagonal blocks $A_{\mathbb{S}\mathds{H}}$ and $A_{\mathds{H}\mathbb{S}}$ are dense --- each consisting of $N_{\mathds{H}}$ full columns or rows --- because the harmonic basis vectors are non-local vector fields.

Naturally, this structure leads to a Schur complement strategy, where the first block row for $\bm{x}_{\mathbb{S}}$ gives
\begin{equation}\label{eq:schurS}
    A_{\mathbb{SS}} \bm{x}_{\mathbb{S}} = \bm{b}_{\mathbb{S}} - A_{\mathbb{S}\mathds{H}}\, \bm{x}_{\mathds{H}}.
\end{equation}
Substituting into the second row yields the Schur complement system for the harmonic unknowns: \vspace*{-0.7em}
\begin{equation}\label{eq:schurH}
    \overbrace{\left(A_{\mathds{HH}} - A_{\mathds{H}\mathbb{S}}\, A_{\mathbb{SS}}^{-1} A_{\mathbb{S}\mathds{H}}\right)}^{S_{\mathds{HH}}} \cdot \bm{x}_{\mathds{H}}
    = \bm{b}_{\mathds{H}} - A_{\mathds{H}\mathbb{S}}\, A_{\mathbb{SS}}^{-1}\, \bm{b}_{\mathbb{S}}.
\end{equation}
The inverse $A_{\mathbb{SS}}^{-1}$ does not need to be formed explicitly, and setting up the Schur complement system \eqref{eq:schurH} requires only $N_{\mathds{H}} + 1$ sparse linear solves with $A_{\mathbb{SS}}$, which is feasible since $N_{\mathds{H}} = b_1(M)$ is typically very small.
The resulting $N_{\mathds{H}} \times N_{\mathds{H}}$ system \eqref{eq:schurH} is dense but very small. Thus, it can be solved directly, after which $\bm{x}_{\mathbb{S}}$ follows from back-substitution via \eqref{eq:schurS}.
Consequently, all large solves involve only the sparse $A_{\mathbb{SS}}$, for which standard solver and preconditioning strategies apply without modification.

\subsection{Hybridization} \label{ssec:hybrid}
A classical approach to reduce the computational costs of discontinuous Galerkin methods is \emph{hybridization} \cite{CGL09,L10,LehrenfeldSchoeberl2016}. 
We briefly discuss an $H(\div)$-conforming HDG formulation of \eqref{eq:surfaceStokes} as introduced in \cite{LedererLehrenfeldSchoeberl2020}.
For an edge $E \in \Eh$, let $\hat{E} \subset \mathbb{R}^2$ be an edge of the reference triangle $\hat{T}$ and $\Phi_E$ be such that $E = \Phi_E(\hat{E})$. Similar as in \eqref{eq:ScalarpolynomialSpaces}, we define the space of tangential vector fields on the edge skeleton as
\begin{align}
    \FacetSp \coloneqq \{ v \cdot \tv \mid~ v \in L^2(\Eh) \text{ s.t. } \forall E \in \Eh \exists \hat{v} \in \hat{\mathbb{P}}^k(\hat{E}) \text{ s.t. } v \vert_{E} = \hat{v} \circ \Phi_E^{-1} \}.
\end{align}
Then, a hybridized version of $a_h(\cdot,\cdot)$ as defined in \eqref{eq:StokesSIP} is given by
\begin{align*}
    a_h^{\text{HDG}}((\u_h,\bm{\lambda}_T),(\v_h, \bm{\mu}_{T})) \coloneqq \sum_{T \in \Th} &\int_{T}\mu\, \underline{\bm{\epsilon}}(\u_h) : \underline{\bm{\epsilon}}(\v_h) \, \mathrm{d}x + \int_{\partial T} \mu\, \underline{\bm{\epsilon}}(\u_h)(\bm{\mu}_T - \v_h) \, \mathrm{d}s \\
    &+ \int_{\partial T} \mu\, \underline{\bm{\epsilon}}(\v_h)(\bm{\lambda}_T - \u_h) \, \mathrm{d}s + \frac{\alpha \mu }{h} \int_{\partial T} (\bm{\lambda}_T - \u_h) (\bm{\mu}_T - \v_h) \, \mathrm{d}s, 
\end{align*}
where $(\u_h, \bm{\lambda}_T), (\v_h, \bm{\mu}_T) \in \BDM^k_0 \times \FacetSp$.
The main computational advantage of this formulation is that the volume unknowns $\u_h$ only couple through the facet unknowns $\bm{\lambda}_T$ and can thus be eliminated through static condensation in a similar manner as \eqref{eq:schurH}. Restricting to the divergence-free subspace $\JBDM$, we obtain the HDG discretization of \eqref{eq:surfaceStokes}: find $(\u_h,\bm{\lambda}_T) \in \JBDM \times \FacetSp$ such that 
\begin{align*}
    a_h^{\text{HDG}}((\u_h,\bm{\lambda}_T),(\v_h, \bm{\mu}_{T})) = f_h(\v_h) \qquad \forall (\v_h, \bm{\mu}_{T}) \in \JBDM \times \FacetSp.
\end{align*}
In the numerical examples below, we also apply this formulation. 
In its Lagrange multiplier form, cf. \eqref{eq:genericproblem:saddlepoint}, this formulation has been used and its computational advantages over the standard DG formulation have been demonstrated in \cite{LedererLehrenfeldSchoeberl2020}. 

The part of the formulation in $\JBDM$ corresponding to the streamfunction can again be translated to a discretization of the biharmonic equation. A related discretization for the flat case based on a $C^0$ hybrid interior penalty method has recently been considered in \cite{DE24C0HHO}.

\subsection{$L^2$ projection into the divergence free subspace of $\BDM$: the Helmholtz projection $\Pi_{\JBDM}$}\label{ssec:helmholtzproj}
A crucial building block for the efficient computation of the harmonic basis in \Cref{alg:harmonic} is the Helmholtz projection $\Pi_{\JBDM}$, which maps any vector field $\bm r_h$ in a parent space, for example, $\BDM_0^k$ or $\mathds{P}^k$, to its divergence-free component in $\JBDM$.
A standard \emph{mixed} formulation of the Helmholtz projection $\Pi_{\JBDM}$ reads: find $(\bm u_h,\lambda_h)\in \mathds{BDM}^k\times \mathbb{P}^{k-1}$ such that
\begin{subequations}\label{eq:mixedproj}
\begin{align}
\scp{\bm u_h}{\bm v_h}_{M_h} + \scp{\div \bm v_h}{ \lambda_h}_{M_h} \ &=\ \scp{\bm r_h}{\bm v_h}_{M_h} && \forall \bm v_h\in \mathds{BDM}_0^k,\\
\scp{\div \bm u_h}{\mu_h}_{M_h}\ &=\ 0 && \forall \mu_h\in \mathbb{P}^{k-1},
\end{align}
\end{subequations}
and then set $\Pi_{\JBDM} \bm r_h := \bm u_h \in \JBDM$.
Note that this mixed formulation actually realizes more than the projection $\Pi_{\JBDM}$: it first projects the input $\bm r_h$ onto $\BDM_0^k$ as $\tilde{\bm r}_h = \Pi_{\BDM_0^k} \bm r_h$ and then splits this projection into its divergence-free part $\bm u_h$ and the discrete gradient part represented by the Lagrange multiplier, $\tilde{\bm r}_h = \bm u_h + \addiv \lambda_h$. We will exploit this decomposition again in \Cref{sec:variants:press} to reconstruct the pressure.

The saddle-point structure of \eqref{eq:mixedproj} admits an efficient realization through a \emph{hybrid mixed formulation}.
Introducing a scalar Lagrange multiplier $\hat\lambda_h \in \SFacetSp$ on the skeleton to enforce normal continuity of $\bm u_h$, both the velocity $\bm u_h$ and the pressure(-like) variable $\lambda_h$ can be eliminated element-by-element by static condensation, since the velocity block is a \emph{local $L^2(M_h)$} operator.
The resulting global system involves only the edge unknowns $\hat\lambda_h$ and constitutes a symmetric positive definite discrete Poisson operator on the skeleton, cf.\ \cite{GuoshengFu2019,CGL09}.
In particular, the overall cost of the solution of \eqref{eq:mixedproj} is comparable to a single scalar Poisson solve on the skeleton and hence 
$\Pi_{\JBDM}$ can be computed efficiently.
This is different to the situation for the discrete Stokes problem, where the velocity block involves the stiffness matrix of the SIP formulation, so that static condensation for only one scalar variable on the edges is not possible. 

\subsection{Pressure reconstruction}\label{sec:variants:press}
By design, the discrete formulations in $\JBDM$ are pressure-free, but there are applications where the pressure field is relevant, for instance, for computing stress forces at the boundary. 
In these situations, pressure can be reconstructed from the velocity. 
Here, we choose a reconstruction based on the space decomposition of the parent space $\mathds{BDM}_0^k$, leading to the same pressure as the one obtained from the mixed formulation in \cite{LedererLehrenfeldSchoeberl2020}.

The key observation is that the load functional $\bm f_h$ is a functional on all of $\BDM^k_0$, not merely on $\JBDM$. By the decomposition of \eqref{eq:BDMDecomp2}, the component of $\bm f_h$ acting on $\JBDM$ drives the velocity solution, while the part acting on $(\JBDM)^\perp = \addiv(\mathbb{P}^{k-1})$ has no effect on velocity --- test functions in $(\JBDM)^\perp$ are annihilated by the divergence-free constraint --- and must therefore be balanced by the pressure gradient.

Concretely, given the velocity solution $\u_h \in \JBDM$ of \eqref{eq:genericproblem}, the force residual $\bm f_h(\v_h) - a_h(\u_h, \v_h)$ vanishes on $\JBDM$ by construction but is in general nonzero on $(\JBDM)^\perp$. The pressure $p_h \in \mathbb{P}^{k-1}$ is determined by requiring that it captures this residual entirely:
\begin{equation}\label{eq:pressurerecon}
    (\div \v_h, p_h)_{M_h} = \bm f_h(\v_h) - a_h(\u_h, \v_h) \qquad \forall \v_h \in (\JBDM)^\perp.
\end{equation}
Since $(\JBDM)^\perp = \addiv(\mathbb{P}^{k-1})$, substituting $\v_h = \addiv q_h$ turns \eqref{eq:pressurerecon} into the \emph{discrete pressure Poisson equation}: find $p_h \in \mathbb{P}^{k-1}$ such that
\begin{equation}\label{eq:pressurepoisson}
    (\addiv p_h, \addiv q_h)_{M_h} = \bm f_h(\addiv q_h) - a_h(\u_h, \addiv q_h) \qquad \forall q_h \in \mathbb{P}^{k-1},
\end{equation}
which is symmetric positive semi-definite and well-posed up to a constant.
An efficient realization of \eqref{eq:pressurepoisson} is obtained by the hybrid mixed formulation of the Helmholtz projection as in \eqref{eq:mixedproj} for $\bm r_h$ being the Riesz representation of the force residual $\bm f_h(\cdot) - a_h(\u_h, \cdot)$ w.r.t. the $L^2(M_h)$ inner product, i.e., replacing $(\bm r_h, \v_h)_{M_h}$ by $\bm f_h(\v_h) - a_h(\u_h, \v_h)$ in \eqref{eq:mixedproj}. The solution $\lambda_h$ then solves \eqref{eq:pressurerecon} and hence we set $p_h := \lambda_h$ to obtain the pressure reconstruction.



\begin{rem}[Pressure reconstruction via incomplete discrete Helmholtz--Hodge decomposition]
    On linear triangulations, one can alternatively employ an incomplete Helmholtz--Hodge decomposition as in \eqref{eq:HodgeNotPk} to reconstruct the pressure.
    Applying the same reasoning yields a discrete pressure Poisson equation analogous to \eqref{eq:pressurepoisson}, except that the operator $\addiv$ is replaced by $\grad_h$, and the space $\mathbb{P}^{k-1}$ is replaced by the corresponding pressure space.
\end{rem} \section{Numerical experiments}\label{sec:numerics}

This section presents two numerical experiments for the unsteady Surface Navier–Stokes equations, illustrating the performance of the streamfunction–harmonic formulation on curved surfaces. All numerical experiments are implemented in the finite element software framework \texttt{Netgen/NGSolve} \cite{Sch97,Sch14} using the \texttt{NGSTrefftz} add-on \cite{NGSTrefftzJOSS} for the (conforming) embedded Trefftz approach. The code is available for reproduction at \cite{grodata}.


\paragraph{Implementation choices}
We use the spatial discretization outlined in \eqref{eq:disc:SNSE} in its hybridized version, as detailed in \Cref{ssec:hybrid}. For the time discretization we use a semi-implicit Euler scheme with constant time step, treating the convection explicitly and the viscosity implicitly; see \cite{LedererLehrenfeldSchoeberl2020} for further details. 
To generate the basis of $\HBDM$, we apply \Cref{alg:harmonic} and use the embedded approach outlined in \Cref{ssec:discinJBDM} to embed $\HBDM$ and $\rot(\mathbb{S}_0^{k+1})$ into $\BDM_0^k$.
The arising sparse linear systems with the same s.p.d.~matrix in each time step are solved with a sparse Cholesky factorization.

\subsection{Scope of numerical investigations and an equivalence validation}\label{sec:sanitycheck}
We omit detailed convergence studies and extensive numerical benchmarks, since the proposed formulation yields a velocity solution that is theoretically equivalent to the formulation in \cite{LedererLehrenfeldSchoeberl2020}. To verify this equivalence numerically, we solve the steady surface Brinkmann problem (Stokes with an added mass term to circumvent killing vector fields) with the smooth body force
$
\bm{f} = (\sin 2x + \cos z,\cos(xy),\sin(x+z))^\top
$
on the torus (genus~$1$, $\dim\HBDM = 2$), using three successively refined meshes and polynomial order~$k=2$.

We compare the velocity-pressure $\mathbf{H}(\operatorname{div})$-HDG formulation of~\cite{LedererLehrenfeldSchoeberl2020} (yielding the solution~$\u_h^{\mathrm{vp}}$) with the streamfunction-harmonic formulation (yielding~$\u_h^{\mathrm{sf}}$). The results are reported in \Cref{tab:equivalence}. On all three meshes, the relative $L^2$ discrepancy
$
\norm{\u_h^{\mathrm{vp}} - \u_h^{\mathrm{sf}}}_{M_h} /
\norm{\u_h^{\mathrm{vp}}}_{M_h}
$
remains several orders of magnitude below the expected discretization error. The small observed differences can be attributed to several factors, including the use of non-identical numerical quadrature rules (on the curved geometry), finite tolerances in the computation of the embedding, and linear solver tolerances.
Overall, the results confirm the discrete equivalence of the two formulations up to the numerical tolerances employed throughout the computation.

\begin{table}[htbp]
  \centering
  \small
  \begin{tabular}{r @{~~~~~~~} r @{~~~~~~~} r @{~~~~~~~} r}
    \toprule
    $N_{\!\mathcal{T}}$
      & ${\norm{\u_h^{\mathrm{vp}} - \u_h^{\mathrm{sf}}}_{M_h}}/%
              {\norm{\u_h^{\mathrm{vp}}}_{M_h}}$
      & $\norm{\div\u_h^{\mathrm{vp}}}_{M_h}$
      & $\norm{\div\u_h^{\mathrm{sf}}}_{M_h}$ \\[4pt]
    \midrule
     $782$ & $3.64\times10^{-8}$ \hspace*{0.85cm} & $2.17\times10^{-16}$ & $3.93\times10^{-11}$ \\
    $1214$ & $3.42\times10^{-9}$ \hspace*{0.85cm} & $3.59\times10^{-15}$ & $8.80\times10^{-11}$ \\
    $4740$ & $2.80\times10^{-8}$ \hspace*{0.85cm} & $3.81\times10^{-15}$ & $1.90\times10^{-09}$ \\
    \bottomrule
  \end{tabular}
  \caption{Numerical equivalence check on the torus (polynomial order $k=2$).
    $\u_h^{\mathrm{vp}}$: solution of velocity-pressure HDG of~\cite{LedererLehrenfeldSchoeberl2020};
    $\u_h^{\mathrm{sf}}$: solution of streamfunction-harmonic formulation (this work). $N_{\!\mathcal{T}}$: number of surface triangles.
  }
  \label{tab:equivalence}
\end{table}

As a consequence of the equivalence we note that the streamfunction-harmonic formulation inherits the same versatility with respect to boundary condition treatment: Dirichlet and stress conditions can be prescribed in the same manner as in the corresponding velocity-pressure formulation, and the corresponding examples with mixed boundary conditions yield identical numerical results up to machine precision. To focus on the role of harmonic fields in topologically non-trivial settings, we concentrate on cases without boundaries or homogeneous Dirichlet boundary conditions, and refer readers to \cite{LedererLehrenfeldSchoeberl2020} for comprehensive boundary-driven test cases.
Likewise, we consider only smooth surfaces, although the methods extend directly to $C^0$ polyhedral surfaces. A rough comparison of the number of degrees of freedom between the velocity-pressure formulation and the streamfunction-harmonic formulation is provided in \Cref{tab:dofs}. 

\begin{table}[!htbp]
    \begin{minipage}{0.2\textwidth}
        \hspace*{-0.15\textwidth}
        \includegraphics[width=1.3\textwidth]{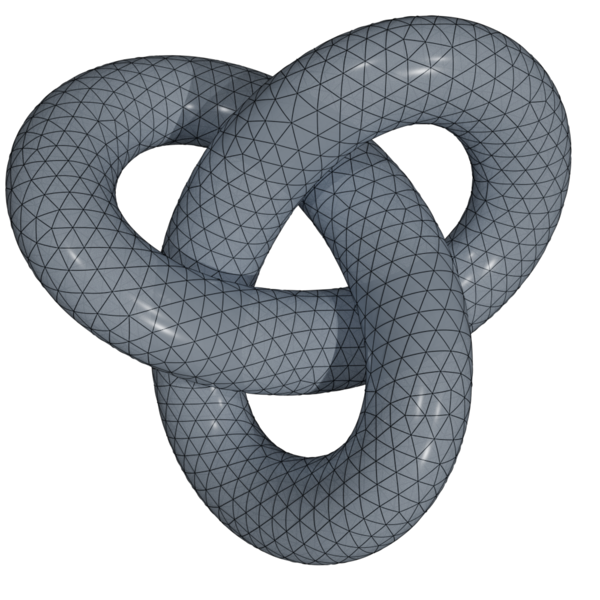} 
        \hspace*{-0.15\textwidth}
        \vspace*{-0.15\textwidth}
    \end{minipage}
    \begin{minipage}{0.799\textwidth}
        \centering
        \begin{tabular}{rr@{ \color{gray}( + }r@{\color{gray}) + }r@{ = }c}
            \toprule
            streamf.-harm. & $\dofs_{\rot}$ & \color{gray}$\dofs_{\lambda}$ & $\dofs_{\!\uharm}$ & $\dofs_{\text{total}}$ \\ 
            \midrule
             & $27920$ & \color{gray}$20940$ &$2$  &$27922 \color{gray}(+20940)$ \\[2ex]
            velocity-press. & $\dofs_{\u}$ & \color{gray}$\dofs_{\lambda}$ & $\dofs_p$ & $\dofs_{\text{total}}$ \\ 
            \midrule
             & $48860$ & \color{gray}$20940$ & $20940$ & $69800 \color{gray}(+20940)$ \\
            \bottomrule
        \end{tabular}
    \end{minipage}
    \vspace{-2.5ex}
    \caption{Left: Geometry and mesh for example of \Cref{sec:trefoil}. Right: Comparison of \dofs~ for the considered streamfunction-harmonic formulation (first row) and the velocity-pressure formulation with $\mathbf{H}(\operatorname{div})$-conforming velocity space (second row). The numbers in gray indicate the contribution of the hybrid unknowns for hybridized version of both methods.} \label{tab:dofs}
\end{table}

\subsection{Trefoil knot}\label{sec:trefoil}
As a first geometrically and topologically non-trivial benchmark, we consider a \emph{thick trefoil knot}. The surface is defined as the set of points at a fixed distance $R \approx 5$ from the centreline of a trefoil knot curve, which is scaled to fit inside the box $[0,100]^3$, see \Cref{tab:dofs} for a visualization of geometry and used mesh. The resulting surface has genus $g=1$, and hence the discrete harmonic space is two-dimensional, $\dim(\HBDM) = 2$.

We consider the unsteady surface Navier--Stokes equations \eqref{eq:surfaceNavierStokes} with kinematic viscosity $\mu = 0.1$.
As a localized forcing we prescribe
$\bm{f} \coloneqq 4 \times 10^{-4}\, \mu \cdot (0,1,0)^\top $ if $x < 40$ and set $\bm{f}$ zero everywhere else. This drives a surface-tangential jet in the $y$-direction in one half of the knot. Although the forcing is localized, the incompressibility constraint enforces a global divergence-free velocity field, so the induced jet is in fact sustained throughout the entire knot.  The initial data is chosen as the Stokes solution, that is, $\bm{u}(\cdot,0) = \bm{u}_{\mathrm{Stokes}}$, producing a maximum velocity of approximately $\|\bm{u}\|_\infty \approx 6$, which yields an estimated Reynolds number $Re = \|\bm{u}\|_\infty 2R \mu^{-1} \approx \frac{6 \cdot 10}{0.1} = 600$. At this Reynolds number, we expect and indeed observe non-laminar behavior on the surface.

\begin{figure}[!htbp]
    \centering
    \vspace*{-1ex}
    \begin{tabular}{c@{~}c@{\!\!}c@{\!\!}c@{\!\!}c@{\!\!}c}
        & $t=0$ & $t=100$ & $t=200$ & $t=400$ & $t=800$
        \\[-0.65ex]
        \rotatebox{90}{\begin{minipage}{2.5cm}\centering\small $\bm{u}$\end{minipage}}&
        \includegraphics[width=0.2\textwidth]{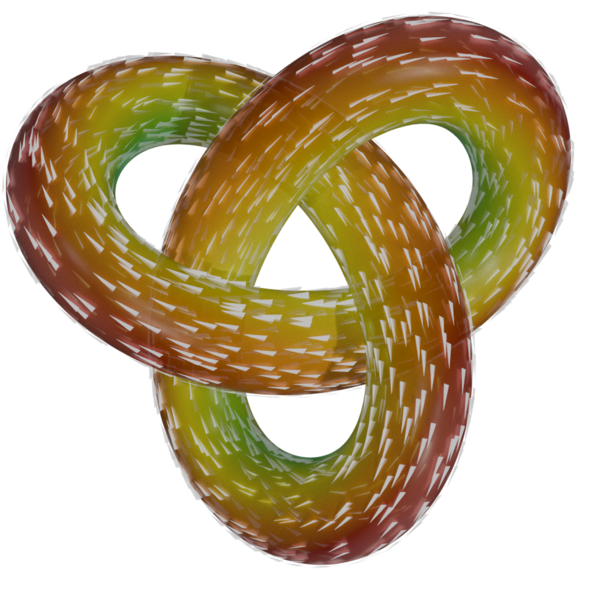}&
        \includegraphics[width=0.2\textwidth]{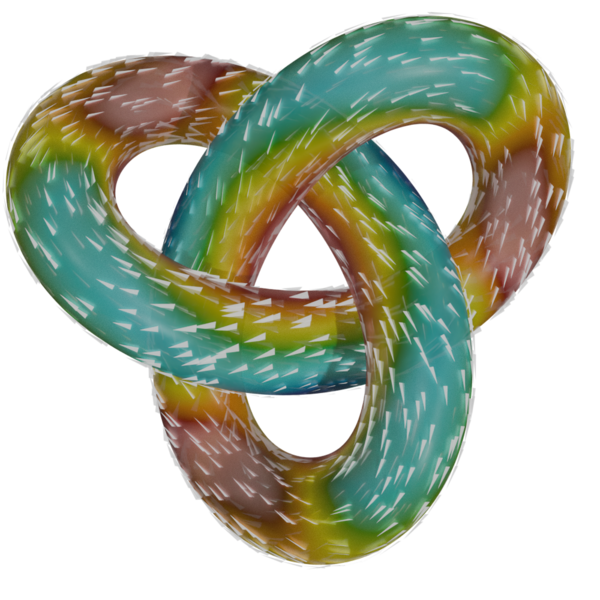}&
        \includegraphics[width=0.2\textwidth]{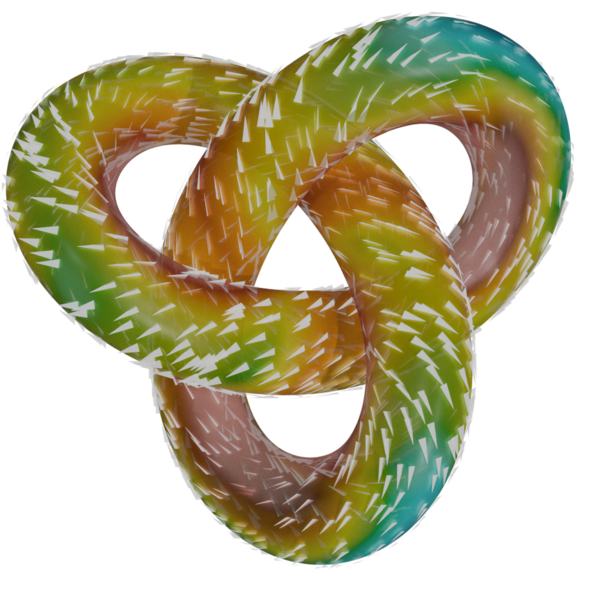}&
        \includegraphics[width=0.2\textwidth]{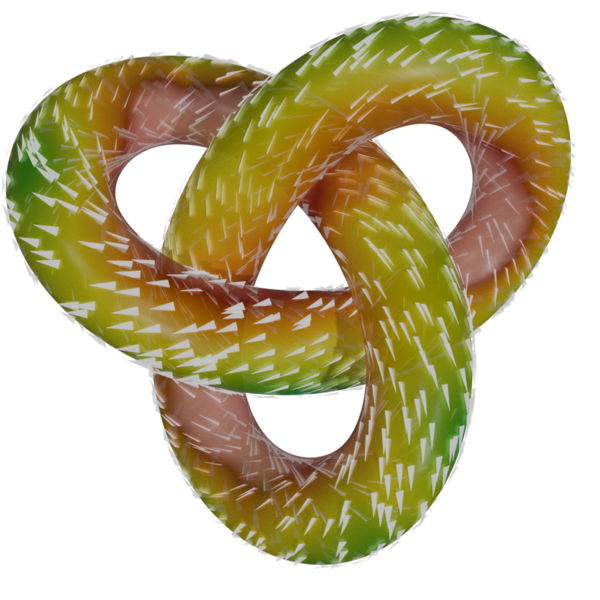}&
        \includegraphics[width=0.2\textwidth]{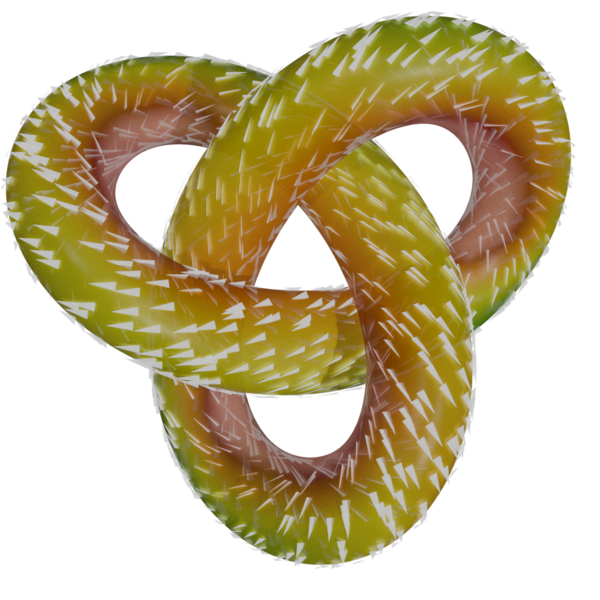}
        \\[-2.2ex]
        \rotatebox{90}{\begin{minipage}{2.5cm}\centering\small $\urot = \rot(\psi)$\end{minipage}}&
        \includegraphics[width=0.2\textwidth]{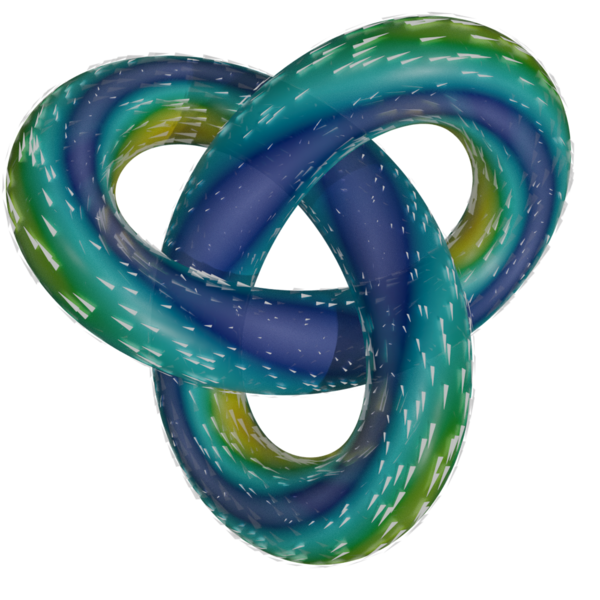}&
        \includegraphics[width=0.2\textwidth]{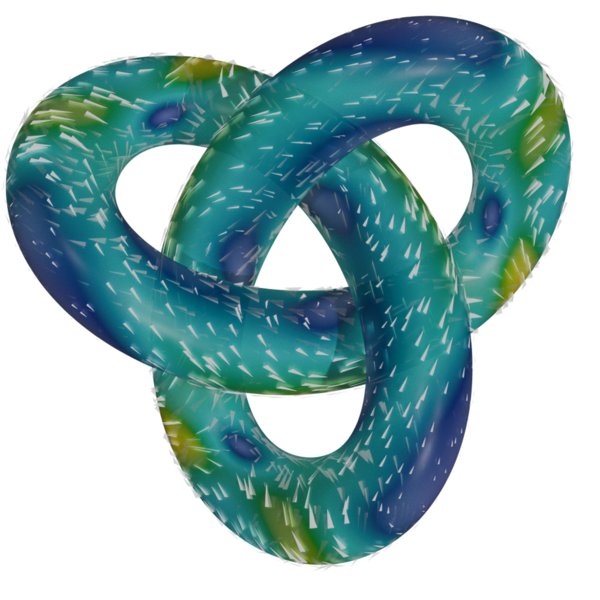}&
        \includegraphics[width=0.2\textwidth]{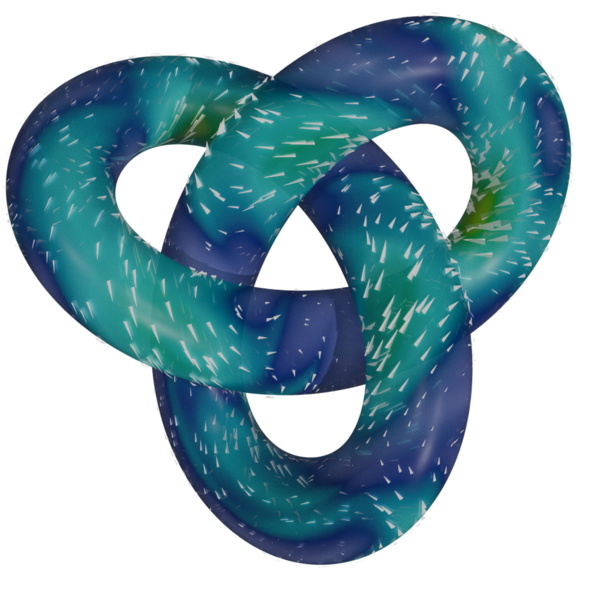}&
        \includegraphics[width=0.2\textwidth]{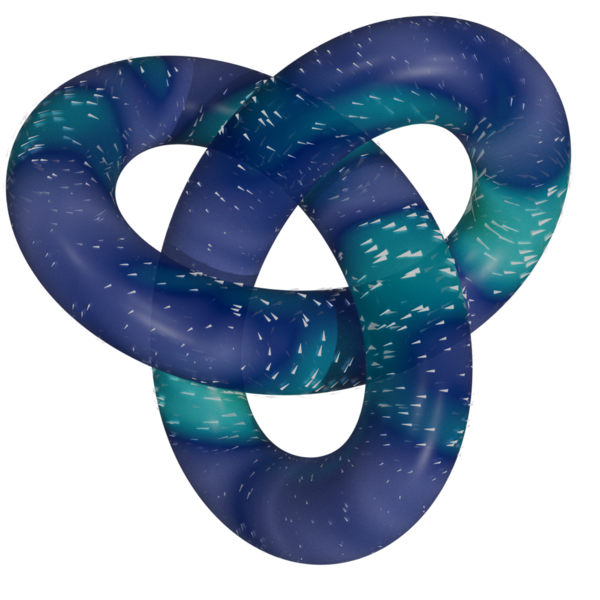}&
        \includegraphics[width=0.2\textwidth]{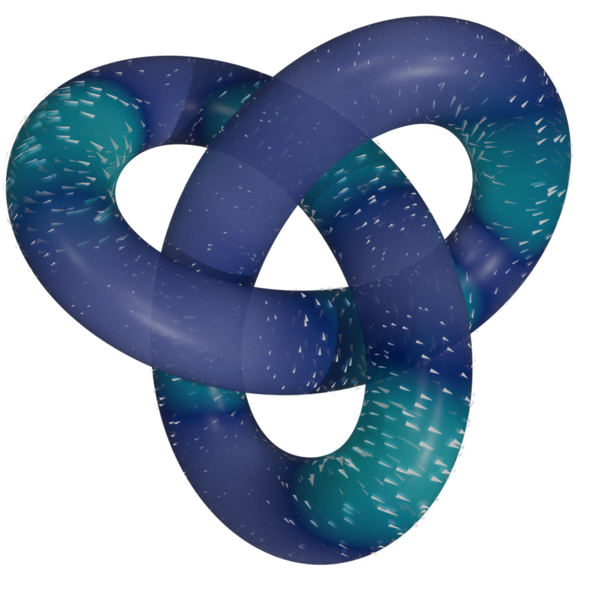}
        \\[-2.2ex]
        \rotatebox{90}{\begin{minipage}{2.5cm}\centering\small $\uharm$\end{minipage}}&
        \includegraphics[width=0.2\textwidth]{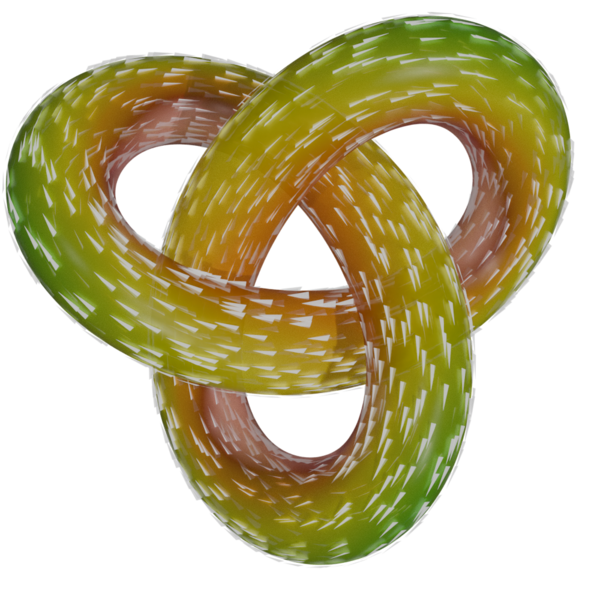}&
        \includegraphics[width=0.2\textwidth]{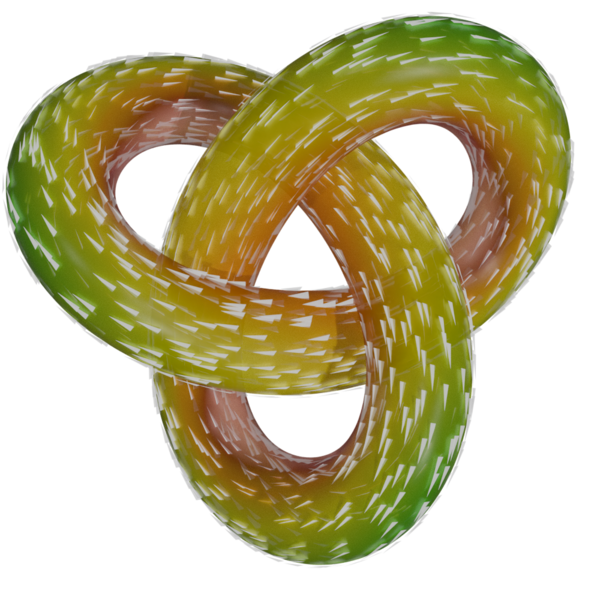}&
        \includegraphics[width=0.2\textwidth]{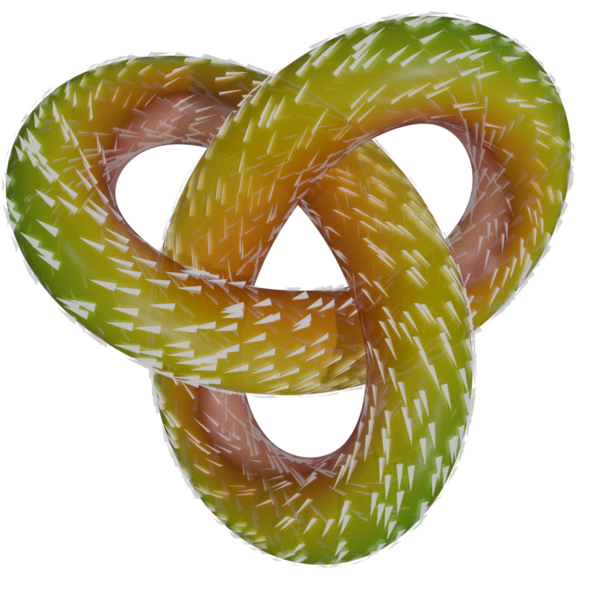}&
        \includegraphics[width=0.2\textwidth]{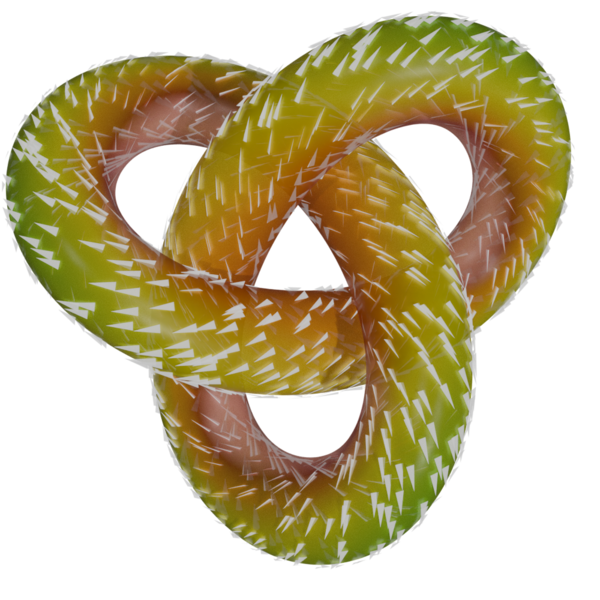}&
        \includegraphics[width=0.2\textwidth]{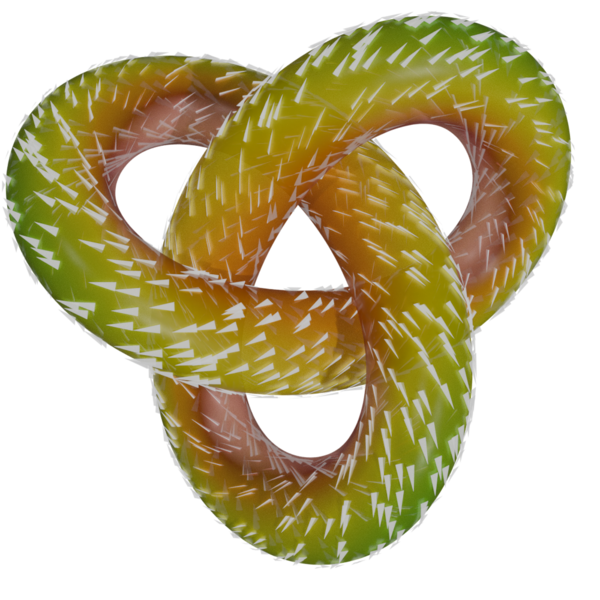}
        \\[-1.5ex]
    \end{tabular}
    \hspace*{0.65\textwidth} \small$0$ \raisebox{1.75ex}{\rotatebox{270}{\includegraphics[width=0.02\textwidth]{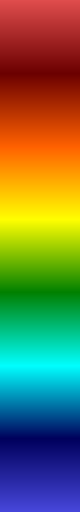}}} $5$
    \vspace*{-0.3cm}
    \caption{Snapshots of the surface Navier--Stokes flow on the thickened trefoil knot at times $t=0,\,100,\,200,\,400,\,800$. Each column shows the total velocity $\bm{u}_h$ (top), the rotational part $\urot = \rot(\psi)$ (middle, arrows scaled $\times 2$), and the harmonic part $\uharm$ (bottom) . The colors indicate velocity magnitude, scaled from $0$ to $5$. A full animation is available in \cite{grodata} and at \href{https://www.youtube.com/watch?v=DqHXLNkhF0o}{\texttt{youtu.be/DqHXLNkhF0o}}.} 
    \label{fig:trefoil}
\end{figure}

We use polynomial degree $k=3$ on an isoparametric quasi-uniform unstructured mesh of $3490$ curved surface triangles and geometric approximation order $k=3$. The approximate mesh size is $h \approx 6$ and the time step size is $\Delta t = 10^{-3}$. The resulting number of degrees of freedom (\dofs) for the streamfunction-harmonic formulation and the velocity-pressure formulation are summarized in \Cref{tab:dofs}.

In \Cref{fig:trefoil}, we display five snapshots of the full velocity $\u$, the rotational part $\urot = \rot(\psi)$, and the harmonic part $\uharm \in \HBDM$ at times $t \in \{0, 100, 200, 400, 800\}$. 

These snapshots reveal two distinct phases of the flow evolution.
In the early phase ($t \lesssim 200$), the flow undergoes a transient rearrangement, where the initial Stokes response to the forcing is destabilized, and  the velocity field reorganizes globally across the knot surface.
Approaching $t \approx 200$, a quasi-organized regime emerges. The dominant coherent structures are vortices that are localized on individual cross-sectional circles of the knot tube and rotate around the tube's cross-section.
These vortex structures are then transported along the knot centerline.
The space splitting makes this structure transparent: the harmonic part $\uharm$ captures the large-scale transport along the knot (as the harmonic fields on a genus-$1$ surface encode the non-contractible circulation around the tube), while the rotational part $\urot = \rot(\psi)$ accounts precisely for the localized azimuthal rotation around each cross-section.
This clean separation between along-knot transport and cross-sectional rotation is a direct consequence of the topological structure of the surface.

\subsection{Pierced \texttt{NGSolve}-sculpture surface}\label{sec:sculpture}

As a second example we consider the unsteady surface Navier--Stokes equations \eqref{eq:surfaceNavierStokes} on the outer surface of the \texttt{NGSolve} sculpture. The geometry consists of a sphere of radius $0.8$ centered at a point $\mathbf{x}_c \in \mathbb{R}^3$, from which three large cylindrical holes of radius $0.4$, aligned with the coordinate axes, and one small cylindrical hole of radius $0.05$ and diameter $0.1$ are removed. A sketch of the geometry is shown in \Cref{fig:sculpturesurf_geo}.

\begin{figure}[!htbp]
    \centering
    \includegraphics[width=0.98\textwidth]{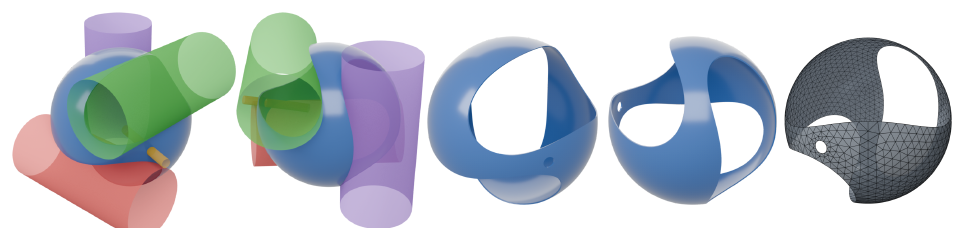}
    \vspace*{-2.5ex}
    \caption{Sketch of the pierced \texttt{NGSolve}-sculpture surface: a sphere with three large cylinders (radius $0.4$) and one small cylinder (radius $0.05$) removed. The left two pictures show the surface with the ``ghost'' cylinders that are removed from two different angles, while the third and fourth picture show the actual surface without the ghost cylinders from the same angles. The right picture shows the surface mesh used for the simulation.}
    \label{fig:sculpturesurf_geo}
\end{figure}
The resulting surface has genus~$0$ and four boundary components (the rims of the cylindrical holes) yielding $\dim(\HBDM)=3$.
Homogeneous (no-slip) boundary conditions are imposed on all four rims, so that we have a surface with boundary $\Gamma \neq \emptyset$.
The small hole was introduced deliberately to trigger a K\'arm\'an vortex street in the wake region downstream of the obstruction.

The driving force is a rigid-body rotation of the sphere around the $z$-axis projected onto the divergence-free subspace,
$\bm{f} \;=\; 100 \mu (\mathbf{x}-\mathbf{x}_c)/\Vert \mathbf{x}-\mathbf{x}_c \Vert \times (0,0,1)^\top$,
with kinematic viscosity $\mu = 10^{-3}$.
The simulation is initialized with the Stokes solution for the same forcing, i.e.\ $\bm{u}(\cdot,0) = \bm{u}_{\mathrm{Stokes}}$, producing a maximum velocity of approximately $\|\bm{u}\|_\infty \approx 2$.
Based on the diameter $D = 0.1$ of the small hole, this gives an estimated Reynolds number $Re = \|\bm{u}\|_\infty D / \mu \approx 200$; in the quasi-stationary regime ($t > 20$) the velocity near the small hole drops to approximately $\|\bm{u}\|_\infty \approx 1$, yielding $Re \approx 100$.

\begin{figure}[!htbp]
    \centering
    \vspace*{-1ex}
    \begin{tabular}{@{}c@{}c@{}c@{}c@{}c@{}c@{}c}
        & $t=0$ & $t=1.25$ & $t=5$ & $t=10$ & $t=20$ & $t=40$
        \\[-0.45ex]
        \rotatebox{90}{\begin{minipage}{2.2cm}\centering\small $\bm{u}$\end{minipage}}&
        \includegraphics[width=0.16\textwidth]{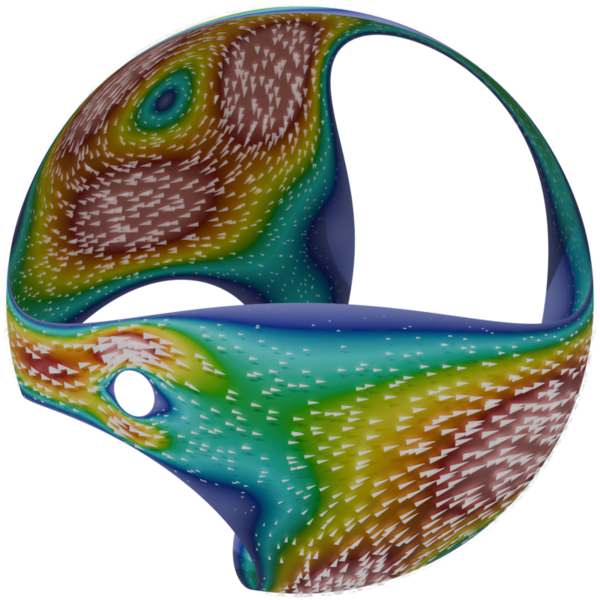}&
        \includegraphics[width=0.16\textwidth]{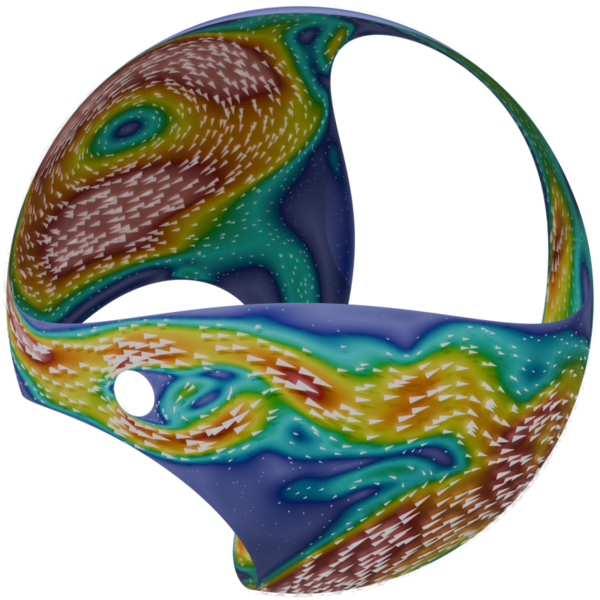}&
        \includegraphics[width=0.16\textwidth]{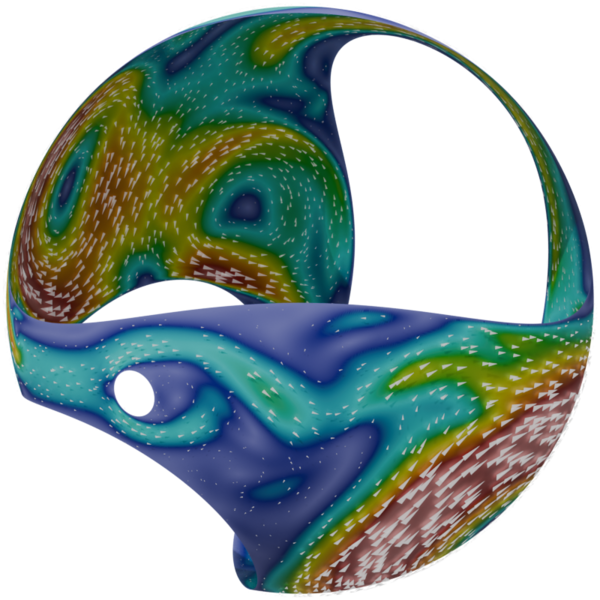}&
        \includegraphics[width=0.16\textwidth]{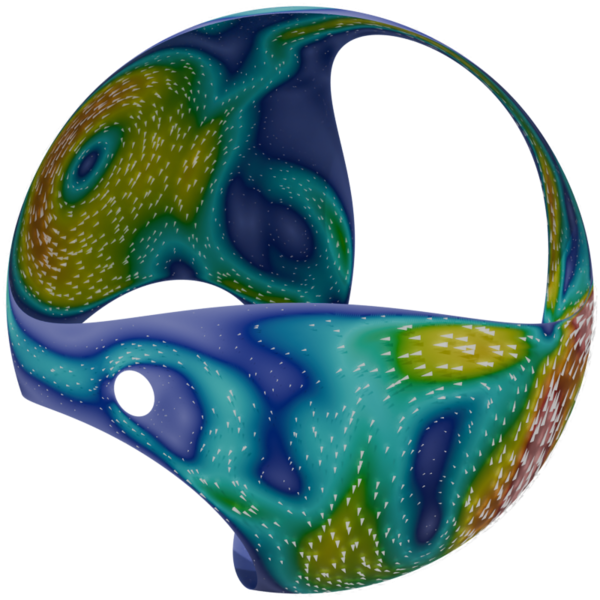}&
        \includegraphics[width=0.16\textwidth]{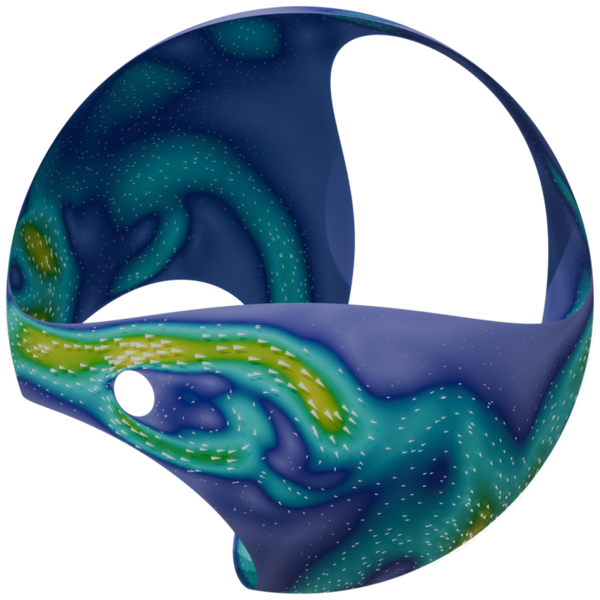}&
        \includegraphics[width=0.16\textwidth]{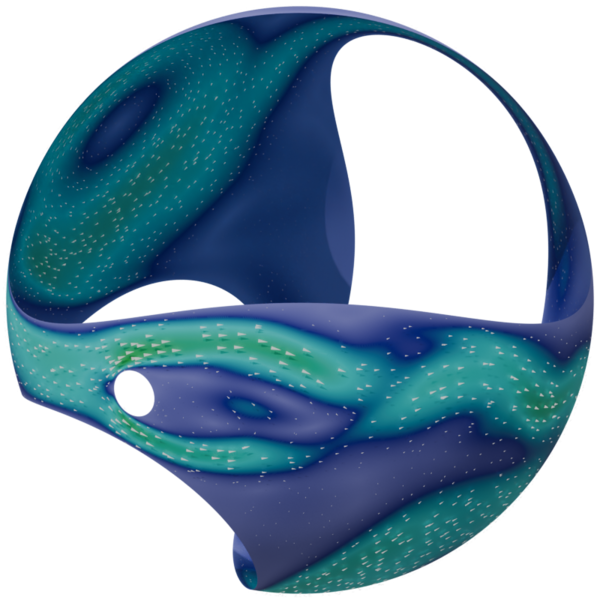}
        \\[-1ex]
        \rotatebox{90}{\begin{minipage}{2.2cm}\centering\small $\urot = \rot(\psi)$\end{minipage}}&
        \includegraphics[width=0.16\textwidth]{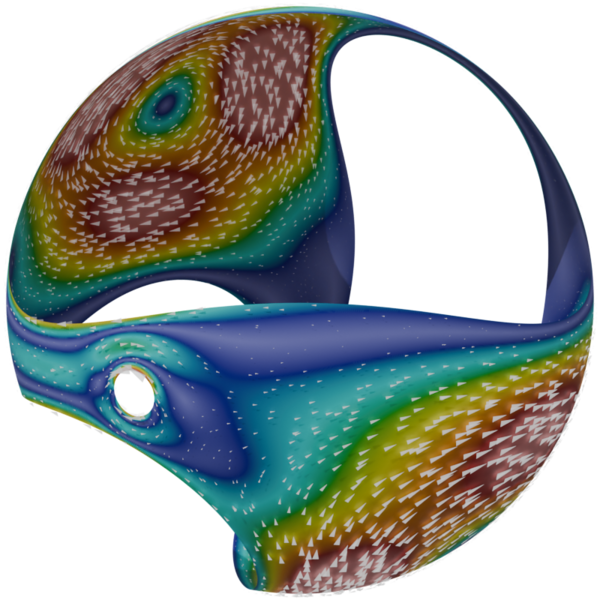}&
        \includegraphics[width=0.16\textwidth]{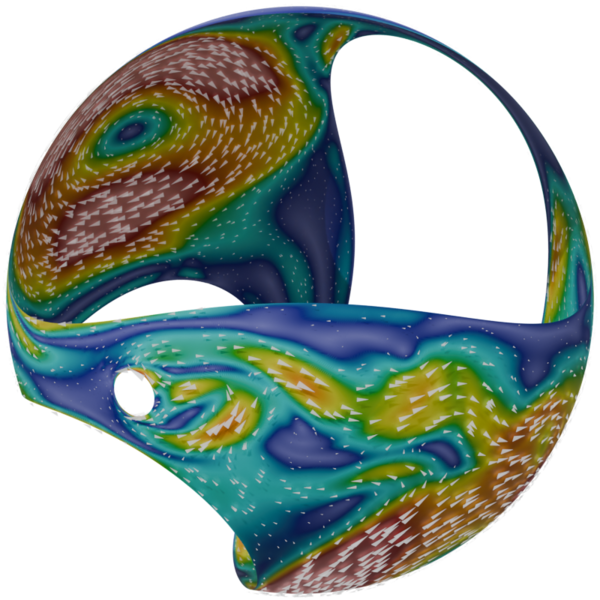}&
        \includegraphics[width=0.16\textwidth]{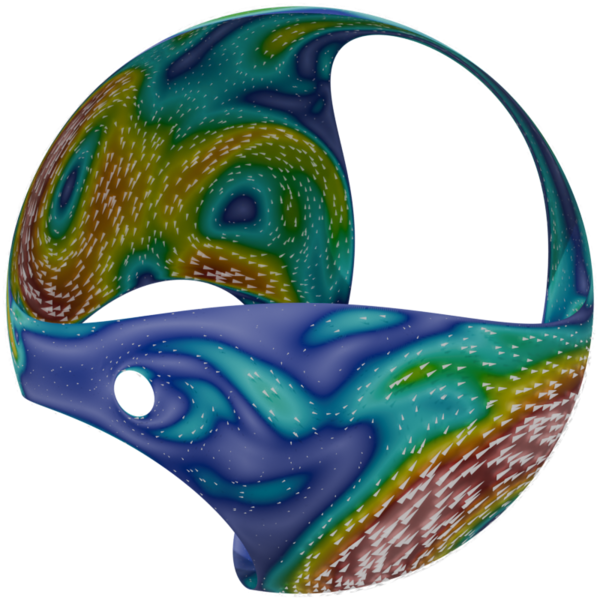}&
        \includegraphics[width=0.16\textwidth]{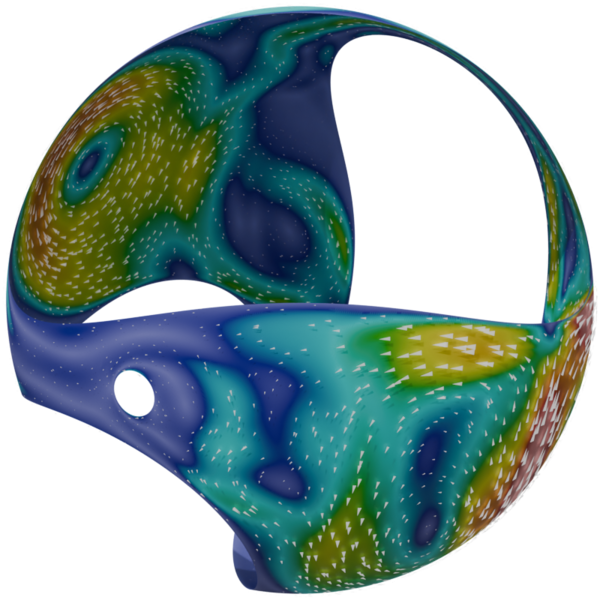}&
        \includegraphics[width=0.16\textwidth]{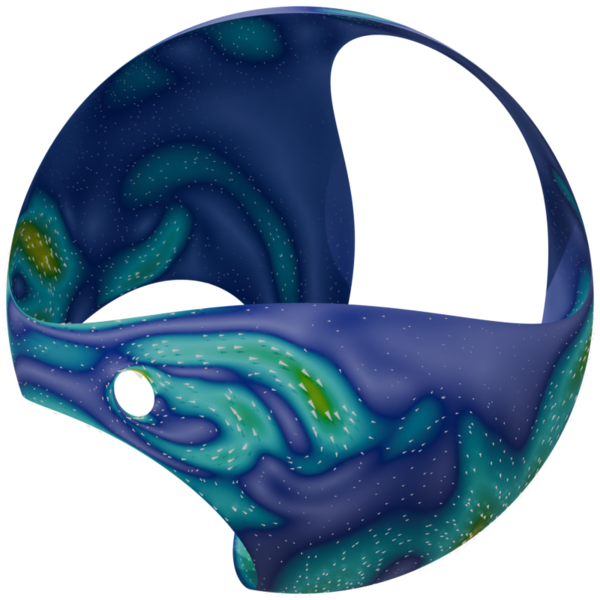}&
        \includegraphics[width=0.16\textwidth]{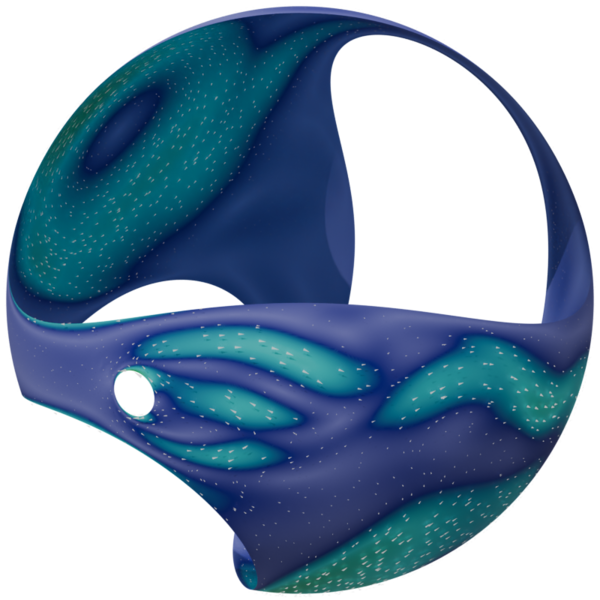}
        \\[-1ex]
        \rotatebox{90}{\begin{minipage}{2.2cm}\centering\small $\uharm$\end{minipage}}&
        \includegraphics[width=0.16\textwidth]{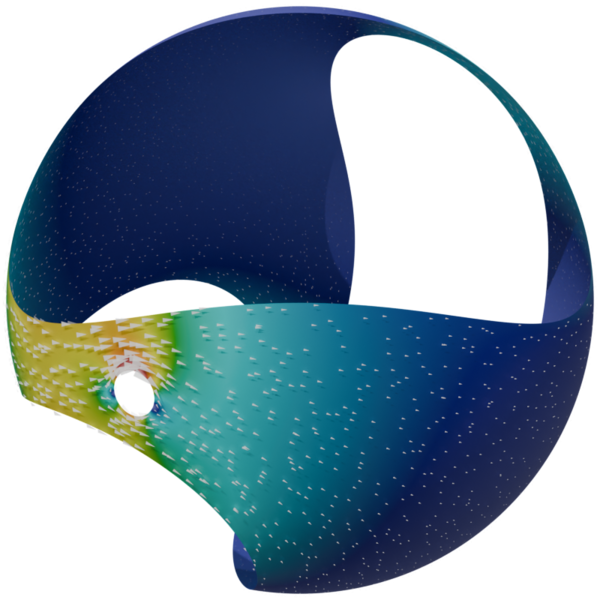}&
        \includegraphics[width=0.16\textwidth]{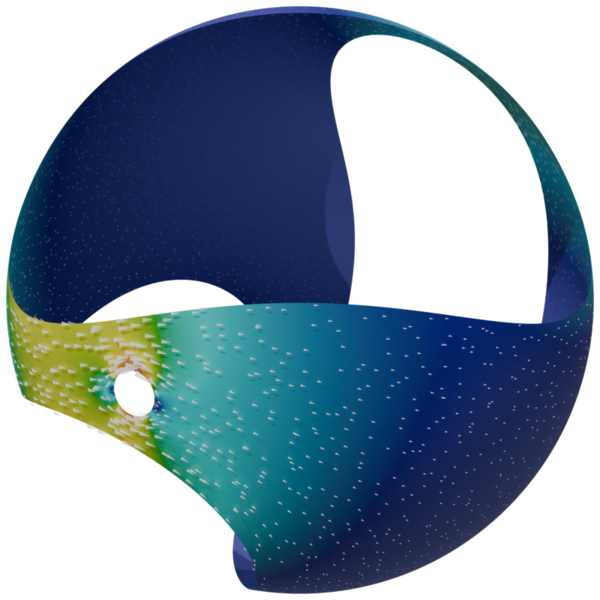}&
        \includegraphics[width=0.16\textwidth]{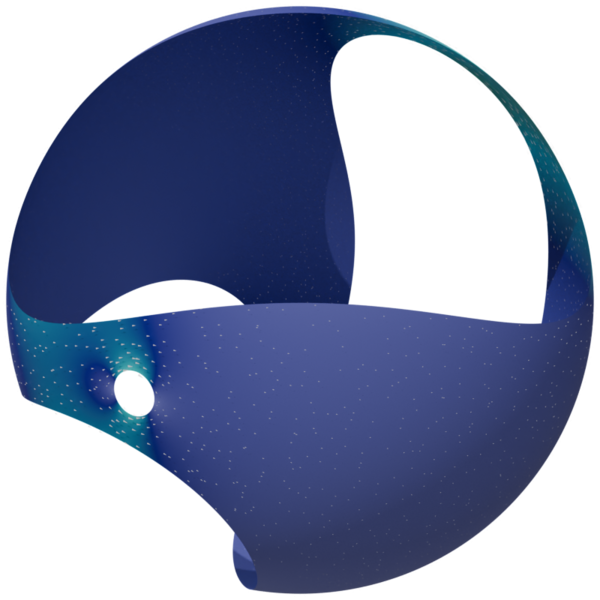}&
        \includegraphics[width=0.16\textwidth]{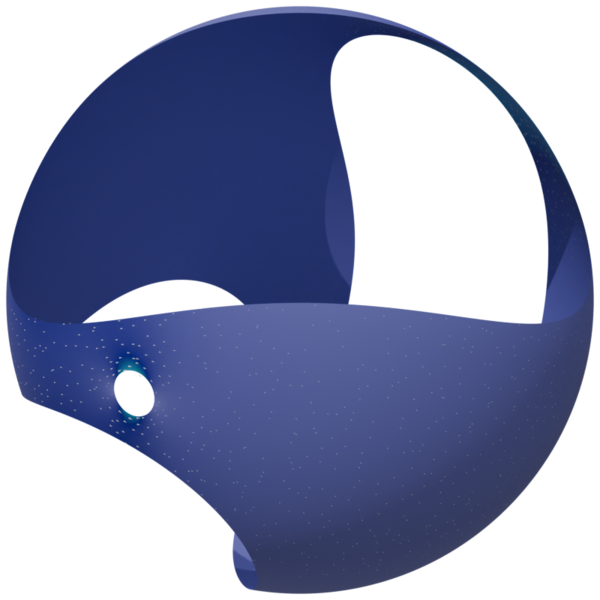}&
        \includegraphics[width=0.16\textwidth]{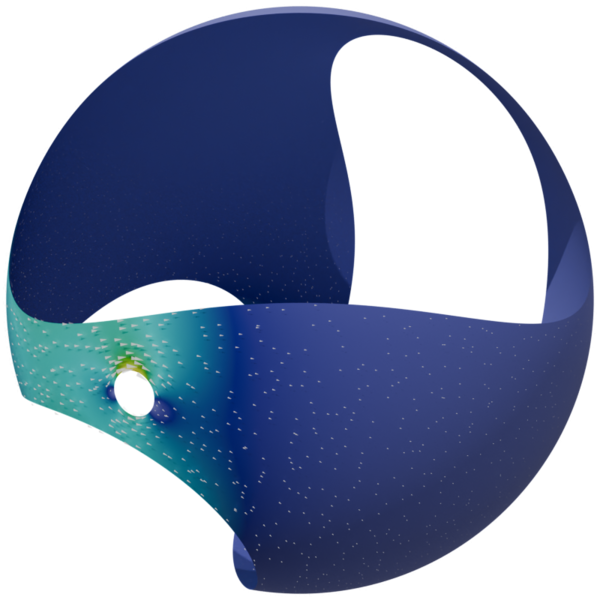}&
        \includegraphics[width=0.16\textwidth]{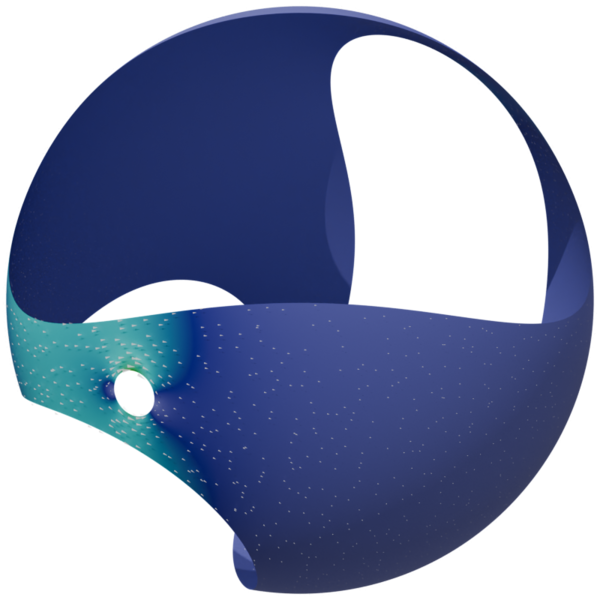}
    \end{tabular}
    \hspace*{0.65\textwidth} \small$0$ \raisebox{1.75ex}{\rotatebox{270}{\includegraphics[width=0.02\textwidth]{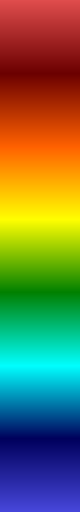}}} $2$
    \vspace*{-0.3cm}
    \caption{Snapshots of the surface Navier--Stokes flow on the pierced sculpture surface at times $t=0,\,1.25,\,5,\,10,\,20,\,40$. Each column shows the total velocity $\bm{u}_h$ (top), the rotational part $\urot = \rot(\psi)$ (middle), and the harmonic part $\uharm$ (bottom). Colors indicate velocity magnitude, scaled from $0$ to $2$. A full animation is available in \cite{grodata} and at \href{https://www.youtube.com/watch?v=fPq8WqddI7M}{\texttt{youtu.be/fPq8WqddI7M}}.}    \label{fig:sculpturesurf}
\end{figure}

We use polynomial degree and geometric approximation order $k=4$ on an isoparametric mesh of $3908$ curved surface triangles with approximate mesh size $h \approx 0.066$.
The streamfunction formulation yields $24,\!928$ streamfunction \dofs~plus $3$ harmonic \dofs. 
For time integration we use the time-step size $\Delta t = 5\times10^{-5}$.

\Cref{fig:sculpturesurf} shows six snapshots of the full velocity $\u$, the rotational part $\urot = \rot(\psi)$, and the harmonic part $\uharm \in \HBDM$ at times $t \in \{0,\, 1.25,\, 5,\, 10,\, 20,\, 40\}$.

We observe three distinct phases of the flow evolution.
Two large vortices form on the two largest free areas of the sphere and exchange flow through the small connecting channels between the cylindrical holes.
Immediately downstream of the small hole, a K\'arm\'an vortex street develops.
During an initial transient phase ($t \lesssim 2$), the flow reorganizes from the Stokes initial condition; this reorganization continues in a less structured form until approximately $t \approx 20$.
For $t > 20$, the flow reaches a quasi-stationary configuration in which the regular K\'arm\'an vortex street behind the small hole is the only remaining fluctuation.

As before, the space decomposition is useful to understand the flow evolution.
The harmonic part $\uharm$ is most significant in two regimes: during the initial transient starting from the Stokes state and in the final quasi-periodic regime with $t > 20$.
Spatially, the harmonic contribution is concentrated on the smallest hole.
In the intermediate reorganization phase ($2 \lesssim t \lesssim 20$), the harmonic field is less important, and the rotational part $\urot = \rot(\psi)$ dominates throughout the simulation, particularly during the reorganization phase ($t < 20$).
For $t > 20$, both the harmonic and rotational contributions become significant, especially in the vicinity of the small hole where the vortex street forms.

\bibliography{literatur}
\bibliographystyle{plain}

\appendix
\section{Abstract Hilbert complexes}\label{sec:HilbertComplexes}
The ideas discussed in this manuscript fit into the abstract framework of Hilbert complexes and finite element exterior calculus (FEEC) as discussed, for instance, in \cite{BL92,AFW06,AFW10,HS12,A18FEEC,H25}. For context, we provide a brief and necessarily incomplete overview relating \Cref{sec:harmonic} to the broader context, and refer to the above references for further details.

%
We say that a sequence $(V^k,A^k)$ forms a \emph{closed Hilbert complex}, if $V^k$ are Hilbert spaces and $A^k : V^{k} \to V^{k+1}$ are closed, densely defined linear operators such that $\ran(A^{k}) \subset \ker(A^{k+1})$ and $\ran(A^{k}) \subset V^{k+1}$ is closed. The $k$-th cohomology space is defined as $\mathcal{H}^k \coloneqq \ker(A^k) / \ran(A^{k-1})$ and the space of harmonic $k$-forms as $\mathfrak{H}^k \coloneqq \ker(A^k) \cap \ran(A^{k-1})^\perp$. A classical result is that the harmonic $k$-forms realize the cohomology of the complex in the sense that $\mathcal{H}^k \cong \mathfrak{H}^k$. Furthermore, the abstract Helmholtz--Hodge decomposition reads as 
\begin{align}\label{eq:AbstractHodge}
    V^k = \ker(A^k) \osum \ker(A^k)^{\perp} = \ran(A^{k-1}) \osum \mathfrak{H}^k \osum \ker(A^k)^{\perp}.
\end{align}
In the setting of \Cref{sec:harmonic}, we have that $A^{k-1} = \rot$, $A^k = \div$, and $\mathfrak{H}^k = \HBDM$. 
Considering the discretization of the Hilbert complex $(V^k,A^k)$, the crucial property to ensure is \emph{cohomology preservation}, i.e. that the spaces of discrete harmonic forms $\mathfrak{H}^k_h$ are isomorphic to the non-discrete ones. Assuming that $(V_h^k,A^k)$, $V_h^k \subset V^k$, is a finite dimensional subcomplex, where the spaces $V_h^k$ fulfill an approximation property, the missing property is the existence of \emph{bounded cochain projections} $\pi_h^k : V^k \to V_h^k$. This means that $\pi_h^k$ is bounded with respect to the $\Vert \cdot \Vert_{V^k}$-norm and the following diagram commutes:
\begin{equation}
    \begin{tikzcd}
    V^{k-1} \arrow[r,"A^{k-1}"] \arrow[d,"\pi_h^{k-1}"] & V^k \arrow[r,"A^k"] \arrow[d,"\pi_h^k"] & V^{k+1} \arrow[d,"\pi_h^{k+1}"]\\
    V_h^{k-1} \arrow[r,"A^{k-1}"] & V_h^k \arrow[r,"A^k"] & V_h^{k+1}
    \end{tikzcd}
\end{equation}
If these conditions are fulfilled, the discrete cohomology is isomorphic to the continuous one, and consequently the Galerkin approximation is consistent, stable, and converges with optimal order. Further, the Helmholtz--Hodge decomposition \eqref{eq:AbstractHodge} holds on the discrete level. 

\subsection{Randomized construction of the harmonic basis for abstract Hilbert complexes} \label{ssec:RandomizedHarmonicAbstractHilbertComplexes}
We can generalize the procedure described in \Cref{sec:harmonic:algo} in the setting of abstract Hilbert complexes. 
To be precise, we may determine a basis for the harmonic $k$-forms $\mathfrak{H}_k \subset V_h^k$ analogously to \Cref{alg:harmonic}. After choosing a random element $r_h \in V_h^{k}$, we project onto the kernel $u_h \coloneqq \Pi_0 r_h \in \ker A^k \subset V_h^{k}$ and compute 
$\psi_h \in V_h^{k-1}$ as the solution to 
\begin{align*}
    (A^{k-1} \psi_h, A^{k-1} \phi_h) = (u_h, A^{k-1} \phi_h) \qquad \forall \phi_h \in V_h^{k-1}
\end{align*}
Then, we set $w_h \coloneqq u_h - A^{k-1} \psi_h$ and orthogonalize.
 \section{Dimensional analysis of discrete harmonic fields}
\label{app:counting}
In this section, we verify that the discrete $\mathds{BDM}$ complex satisfy the topological requirements necessary to preserve the cohomology of the continuous complex. We further discuss possible discretizations of the $H^2$-complex in the planar case, as already discussed in \Cref{rem:H2complex}.

Let $(\mathcal{V}, \mathcal{E}, \mathcal{T})$ denote the sets of vertices, edges, and triangles of the discrete surface $M_h$.
We partition the vertices and edges into interior ($\mathcal{V}_{\mathcal{I}}, \mathcal{E}_{\mathcal{I}}$) and boundary ($\mathcal{V}_{\Gamma_h}, \mathcal{E}_{\Gamma_h}$) subsets. 
Since the boundary consists of closed loops, we have $|\mathcal{V}_{\Gamma_h}| = |\mathcal{E}_{\Gamma_h}|$. Furthermore, the Euler characteristic is given by $\chi(M) = |\mathcal{V}| - |\mathcal{E}| + |\mathcal{T}|$.

\subsection{The $\mathds{BDM}$-complex}
The following result establishes that the space $\mathds{H}_{\mathds{BDM}}^k$ correctly captures the topology of the manifold, consistent with \cite[Prop. 2.2]{KanschatSharma2014}.

\begin{xthm}[Dimension of $\mathds{H}_{\mathds{BDM}}^k$]
\label{thm:countBDM}
    The space of discrete harmonic fields $\mathds{H}_{\mathds{BDM}}^k$ satisfies $\mathrm{dim}\left(\mathds{H}_{\mathds{BDM}}^k\right) = b_1(M)$.
\end{xthm}

\begin{proof}
    By construction, it holds that $\mathrm{dim}(\mathds{H}_{\mathds{BDM}}^k) = \mathrm{dim}(\mathds{J}_{\mathds{BDM}}^k)-\mathrm{dim}(\rot(\mathbb{S}^{k+1}_0))$.
    The dimensions of these contributions are
    \begin{align*}
        \mathrm{dim}(\mathds{J}_{\mathds{BDM}}^k)&=(k+1)|\mathcal{E}_{\mathcal{I}}| + \frac{k(k-1)}{2}|\mathcal{T}|-(|\mathcal{T}|-1), \\
        \mathrm{dim}(\rot(\mathbb{S}^{k+1}_0)) &= |\mathcal{V}_{\mathcal{I}}| + k|\mathcal{E}_{\mathcal{I}}|+\frac{k(k-1)}{2}|\mathcal{T}|-\delta_{\Gamma, \emptyset},
    \end{align*}
    where $\delta_{\Gamma,\emptyset}=1$ if $\Gamma=\emptyset$ and $0$ otherwise. 
    Subtracting these dimensions, and using $|\mathcal{E}_{\Gamma_h}|=|\mathcal{V}_{\Gamma_h}|$ along with the definition of $\chi(M)$, yields
    \begin{align*}
        \mathrm{dim}(\mathds{H}_{\mathds{BDM}}^k) &= |\mathcal{E}_{\mathcal{I}}| - |\mathcal{V}_{\mathcal{I}}| - |\mathcal{T}| + 1 + \delta_{\Gamma, \emptyset}= 1 + \delta_{\Gamma, \emptyset} - \chi(M).
    \end{align*}
    Finally, employing the Euler--Poincaré formula $\chi(M)=b_0-b_1+b_2$ and observing that for a connected surface $b_0=1$ and $b_2=\delta_{\Gamma, \emptyset}$, the result follows immediately.
\end{proof}

\subsection{The $\mathds{SV}$-complex}
\label{sec:SVcomplex}
We restrict our analysis to the planar setting $M_h \subset \mathbb{R}^2$, assuming that $M_h$ contains no singular vertices. Analogously to the $\BDM$ construction, we define the space of divergence-free velocity fields as $\mathds{J}_{\mathds{SV}}^k\coloneqq \mathds{SV}^k_0 \cap \mathrm{ker}(\mathrm{div})$, where $\mathds{SV}^k_0$ is the Scott-Vogelius finite element space of degree $k$, with vanishing normal component on $\Gamma_h$.

For $k \geq 4$, the divergence-free subspace of $\mathds{SV}^k_0$ has the dimension 
\begin{align*}
    \mathrm{dim}(\mathds{J}_{\mathds{SV}}^k)&=\vert\mathcal{V}\vert+\vert\mathcal{V}_{\mathcal{I}}\vert+(k-1)(\vert \mathcal{E}\vert +\vert \mathcal{E}_{\mathcal{I}}\vert)+(k-1)(k-2)\vert \mathcal{T}\vert-\left(\frac{k(k+1)}{2}\vert \mathcal{T}\vert-1\right). 
\end{align*}

As noted in \Cref{rem:H2complex}, the Argyris element $\mathbb{A}_0^{k+1}$ with zero boundary trace is unsuitable as a streamfunction space for these velocity fields. This incompatibility can be established simply by calculating the degrees of freedom for the Argyris space with a vanishing boundary trace, which yields:
\begin{align*}
    \mathrm{dim}(\mathrm{rot}(\mathbb{A}_0^{k+1})) =& \hphantom{+} 3\vert\mathcal{V}\vert+3\vert \mathcal{V}_{\mathcal{I}}\vert+(k-3)\vert \mathcal{E}\vert+(k-4)\vert \mathcal{E}_{\mathcal{I}}\vert+\frac{(k-3)(k-4)}{2} \vert \mathcal{T}\vert.
\end{align*}
Investigating the difference $\mathrm{dim}(\mathbb{J}_{\mathbb{SV}}^k)-\mathrm{dim}(\mathrm{rot}(\mathbb{A}_0^{k+1}))$, we observe that it scales with $-4\vert \mathcal{V}_{\mathcal{I}}\vert$, neglecting lower-order boundary contributions.

To correctly capture the topology of the domain, this dimension difference must relate to the Euler characteristic $\chi(M_h) = \vert \mathcal{V}\vert - \vert \mathcal{E}\vert + \vert \mathcal{T}\vert$. While a contribution of $-1\vert \mathcal{V}_{\mathcal{I}}\vert$ is strictly necessary to reconstruct the $\chi(M_h)$ term, the Argyris element leaves an unresolvable excess of $-3\vert \mathcal{V}_{\mathcal{I}}\vert$. 

Indeed, this obstruction is not merely theoretical; for simple triangulations, one can explicitly calculate the dimension difference and observe that it fails to match the Betti number $b_1(M)$. Consequently, the Argyris element leads to a discrete complex that incorrectly captures the cohomology of the manifold.

\paragraph{The Hsieh--Clough--Tocher element} To correctly define the subspace of harmonic fields within $\mathds{J}_{\mathds{SV}}^k$, we employ the space $\mathbb{HCT}^{k+1}_0$ with zero boundary trace (where $k\geq 2$), defined on a Clough--Tocher (Alfeld) split.

Since $\rot(\mathbb{HCT}^{k+1}_0) \subset \mathds{J}_{\mathds{SV}}^k$, we define the space of discrete harmonic fields $\mathds{H}_{\mathds{SV}}^k$ as the orthogonal complement of $\rot(\mathbb{HCT}_0^{k+1})$ within the velocity space
\begin{equation*}
    \mathds{H}_{\mathds{SV}}^k \coloneqq \left(\rot(\mathbb{HCT}^{k+1}_0)\right)^\perp \subset \mathds{J}_{\mathds{SV}}^k,
\end{equation*}
where the orthogonality is defined with respect to the $L^2(M_h)$ inner product.

The fact that $\mathds{H}_{\mathds{SV}}^k$ correctly captures the topology of $M$ is established by the following theorem.

\begin{xthm}[Dimension of $\mathds{H}_{\mathds{SV}}^k$]
\label{thm:countHCT}
    The space of discrete harmonic fields $\mathds{H}_{\mathds{SV}}^k$ satisfies $\mathrm{dim}\left(\mathds{H}_{\mathds{SV}}^k\right)=b_1(M)$.
\end{xthm}

\begin{proof}
     By construction, it holds that $\mathrm{dim}(\mathds{H}_{\mathds{SV}}^k)=\mathrm{dim}(\mathds{J}_{\mathds{SV}}^k)-\mathrm{dim}(\rot(\mathbb{HCT}_0^{k+1}))$.
    Counting the \dofs \ on the underlying macro-elements, the dimensions of these contributions are \cite{DouglasDupontPercellScott79}: 
    \begin{align*}
        \mathrm{dim}(\mathds{J}_{\mathds{SV}}^k)=&(\vert\mathcal{V}\vert+\vert\mathcal{V}_{\mathcal{I}}\vert+2\vert \mathcal{T}\vert)+(k-1)(\vert \mathcal{E}\vert+\vert \mathcal{E}_{\mathcal{I}}\vert+6\vert \mathcal{T}\vert)
        \\&+(k-1)(k-2)\cdot3\vert \mathcal{T}\vert-\left(\frac{k(k+1)}{2}\cdot3\vert \mathcal{T}\vert-1\right),
        \\\mathrm{dim}(\rot(\mathbb{HCT}_0^{k+1}))=&\vert \mathcal{V}\vert+2\vert \mathcal{V}_{\mathcal{I}}\vert+(k-1)\vert \mathcal{E}\vert+(k-2)\vert\mathcal{E}_{\mathcal{I}}\vert+3\frac{(k-1)(k-2)}{2}\vert \mathcal{T}\vert.
    \end{align*}
    Subtracting these dimensions, and grouping the terms by mesh entity, yields
    \begin{align*}
         \mathrm{dim}(\mathds{J}_{\mathds{SV}}^k)-\mathrm{dim}(\rot(\mathbb{HCT}_0^{k+1}))=&-\vert \mathcal{V}_{\mathcal{I}}\vert+\vert \mathcal{E}_{\mathcal{I}}\vert+1
         \\&+\left(2+6(k-1)+3\frac{(k-1)(k-2)}{2}-3\frac{k(k+1)}{2}\right)\vert \mathcal{T}\vert
        \\=&-\vert \mathcal{V}\vert+\vert \mathcal{E}\vert-\vert \mathcal{T}\vert+1=1-\chi(M).
    \end{align*}
    For the final steps we use the relation $\vert \mathcal{E}_{\Gamma_h}\vert=\vert \mathcal{V}_{\Gamma_h}\vert$ and the definition of $\chi(M)$.
    Finally, employing the Euler--Poincaré formula $\chi(M)=b_0-b_1+b_2$ and observing that for a connected surface with boundary $b_0=1$ and $b_2=0$, the result follows immediately. Similar computations for simply connected domains can be found in \cite[Sec.~4.3]{JLMNR17}.
\end{proof} 
\end{document}